\documentclass[11pt,a4paper,usenames,dvipsnames,reqno]{amsart}
\usepackage{amssymb}
\usepackage{amscd}
\usepackage{amsfonts}
\usepackage[margin=1in]{geometry}
\usepackage{accents}

\usepackage[all,arc]{xy}
\usepackage{enumerate}
\usepackage{mathrsfs}
\usepackage{color}   
\usepackage{hyperref}
\hypersetup{
    colorlinks=true, 
    linktoc=all,     
    linkcolor=blue,  
}
\usepackage{amsthm}
\usepackage{amsmath}
\usepackage{graphicx}
\usepackage{wasysym}
\usepackage{pgf,tikz}
\usetikzlibrary{positioning}
\usepackage{ marvosym }
\usepackage{tikz-cd}
\usepackage{ textcomp }
\usepackage{bbm}
\usepackage{ stmaryrd }
\usepackage{appendix}

\newcommand{\cc}{\mathbb C}

\newcommand{\zz}{\mathbb Z}

\newcommand{\rr}{\mathbb R}
\newcommand{\A}{\mathbb A}

\newcommand{\la}{\langle}
\newcommand{\ra}{\rangle}
\newcommand{\lra}{\longrightarrow}
\newcommand{\hra}{\hookrightarrow}

\newcommand{\al}{\alpha}
\newcommand{\be}{\beta}
\newcommand{\ga}{\gamma}

\newcommand{\ep}{\epsilon}

\newcommand{\lam}{\lambda}
\newcommand{\Lam}{\Lambda}
\newcommand{\ka}{\kappa}

\DeclareMathOperator{\ad}{ad}

\DeclareMathOperator{\End}{End}

\DeclareMathOperator{\GL}{GL}

\DeclareMathOperator{\Hom}{Hom}

\DeclareMathOperator{\PGL}{PGL}

\DeclareMathOperator{\SO}{SO}
\DeclareMathOperator{\SL}{SL}
\DeclareMathOperator{\Sp}{Sp}

\DeclareMathOperator{\Lie}{Lie}
\DeclareMathOperator{\sspan}{span}
\DeclareMathOperator{\diag}{diag}

\DeclareMathOperator{\Spec}{Spec}
\DeclareMathOperator{\Stab}{Stab}
\DeclareMathOperator{\codim}{codim}
\DeclareMathOperator{\corank}{corank}
\DeclareMathOperator{\rk}{rank}
\DeclareMathOperator{\Ad}{Ad}

\newcommand{\fg}{\mathfrak g}
\newcommand{\fgl}{\mathfrak gl}
\newcommand{\ft}{\mathfrak t}
\newcommand{\fk}{\mathfrak k}
\newcommand{\fa}{\mathfrak a}
\newcommand{\fc}{\mathfrak c}
\newcommand{\fu}{\mathfrak u}

\newcommand{\fb}{\mathfrak b}
\newcommand{\fn}{\mathfrak n}
\newcommand{\fd}{\mathfrak d}
\newcommand{\fl}{\mathfrak l}
\newcommand{\fp}{\mathfrak p}
\newcommand{\fq}{\mathfrak q}
\newcommand{\fz}{\mathfrak z}

\newcommand{\fsl}{\mathfrak{sl}}

\newcommand{\fso}{\mathfrak{so}}

\newcommand{\calT}{\mathcal{T}}

\newcommand{\calo}{\mathcal{O}}

\newcommand{\caln}{\mathcal N}

\newcommand{\fgres}{\widetilde{\fg}_1^{res}}
\newcommand{\fgreg}{\widetilde{\fg}_1^{reg}}
\newcommand{\fgbul}{\widetilde{\fg}_1}

\newcommand{\Gm}{\mathbb{G}_m}

\newcommand{\iso}{\xrightarrow{\sim}}

\newcommand{\Fl}{\mathcal{F}l}

\makeatletter
\def\Ddots{\mathinner{\mkern1mu\raise\p@
\vbox{\kern7\p@\hbox{.}}\mkern2mu
\raise4\p@\hbox{.}\mkern2mu\raise7\p@\hbox{.}\mkern1mu}}
\makeatother

\newtheorem{Thm}{Theorem}[section]
\newtheorem{Prop}[Thm]{Proposition}
\newtheorem{Lem}[Thm]{Lemma}
\newtheorem{Cor}[Thm]{Corollary}

\newtheorem{Conj}[Thm]{Conjecture}

\theoremstyle{definition}
\newtheorem{Def}[Thm]{Definition}

\theoremstyle{remark}
\newtheorem{Rem}[Thm]{Remark}
\newtheorem{Ex}[Thm]{Example}

\theoremstyle{definition}

\title[Grothendieck-Springer for symmetric spaces]{An analogue of the Grothendieck-Springer resolution for symmetric spaces}
\author{Spencer Leslie}

\date\today

\address{Department of Mathematics, Duke University, 120 Science Drive, Durham, NC, USA}

\email{lesliew@math.duke.edu}


\begin{document}

\begin{abstract}
Motivated by questions in the study of relative trace formulae, we construct a generalization of Grothendieck’s simultaneous resolution over the regular locus of certain symmetric pairs. We use this space to prove a relative version of results of Donagi-Gaitsgory about the automorphism sheaf of regular stabilizers. We also obtain partial results toward applications
in Springer theory for symmetric spaces.
\end{abstract}

\maketitle


Let $G$ be a connected reductive group over an algebraically closed field $k$, and let $\fg$ denote its Lie algebra. We assume throughout that the characteristic of $k$ is zero or sufficiently large with respect to $G$. An important construction in the representation theory of $\fg$ is the \emph{simultaneous resolution of singularities} of Grothendieck $$\widetilde{\fg}=\{(X,B)\in \fg\times\Fl_G: X\in \Lie(B)\},$$ where $\Fl_G$ is the flag variety of Borel subgroups of $G$. This space plays a central role in Springer theory, where one needs both the property that it simultaneously resolves the singularities of the quotient map with respect to the adjoint action $\chi:\fg\to \fg//G$, and the existence of the Cartesian diagram
\begin{equation}\label{intro diagram}
\begin{tikzcd}
 \widetilde{\fg}^{reg}\ar[d,"{\pi}"]\ar[r,"\widetilde{\chi}"]&\ft\ar[d]\\
\fg^{reg}\ar[r,"\chi"]&\ft/W,
\end{tikzcd}
\end{equation}
where $W$ is the Weyl group acting on a Cartan subalgebra $\ft$, $\pi:\widetilde{\fg}\to \fg$ is the projection, and we have made use of the Chevalley isomorphism $\fg//G\cong \ft/W$. This diagram may be used to induce Springer's $W$-action on the cohomology of Springer fibers.

 The variety $\widetilde{\fg}$ also arises in the theory of $G$-Higgs bundles as studied by Donagi and Gaitsgory. In \cite{DonGaits}, the authors identify abstract Hitchin fibers as a gerbe over a certain abelian group scheme which acts on the Hitchin fibration. In their analysis, the restriction of the Grothendieck-Springer resolution to the regular locus of $\fg$ is used to compare the moduli space of \emph{regular centralizers} with the moduli space of \emph{regular orbits} of $\fg$. In his study of the Langlands-Shelstad fundamental lemma, Ng\^{o} \cite{Ngo06} utilized this connection in an important way. One of the goals of this present article is to establish an analogous statement in the case of a symmetric space (see Theorem \ref{Thm: local etale}).

More precisely, assume now that $G$ admits an involutive automorphism $\theta:G\to G$, and let $G_0$ be the fixed-point subgroup of $\theta$. The pair $(G,G_0)$ is called a symmetric pair. Passing to the Lie algebra $\fg=\Lie(G)$, the differential of $\theta$ (which we also denote by $\theta:\fg\to \fg$) produces the decomposition 
\[
\fg=\fg_0\oplus \fg_1,
\]
where $\fg_i$ is the $(-1)^i$-eigenspace of $\theta$. Then $G_0$ acts on the \emph{infinitesimal symmetric space} $\fg_1$ by restriction of the adjoint action. Studying the $G_0$-orbits on $\fg_1$ gives a natural generalization of the adjoint representation. In fact, the adjoint representation may be recovered by considering the involution of $\fg\oplus \fg$ given by swapping the two factors. 

\subsection{An analogue of the Grothendieck-Springer resolution} In this paper, we construct and study a generalization of Grothendieck's resolution for the quotient of $\fg_1$ by the action of $G_0$ over the regular locus of $\fg_1$ under the assumption that $\theta$ is \emph{quasi-split}. This is equivalent to the existence of a Borel subgroup $B\subset G$ such that $B\cap \theta(B)$ is a torus. In this setting, we define a sub-scheme $\widetilde{\fg}_1\subset \fg_1\times_\fg\widetilde{\fg}$ and, setting $\fgreg=\fgbul\times_{\fg_1}\fg_1^{reg}$, we prove that the induced map $\pi: \widetilde{\fg}^{reg}_1\to \fg_1^{reg}$ behaves like an analogue of Grothendieck's resolution:
\begin{Thm}\label{Thm: main theorem intro}
Let $(\fg,\fg_0)$ be a quasi-split symmetric pair with $\fg=\fg_0\oplus \fg_1$. There is a closed subscheme $\fgbul\subset \fg_1\times_\fg\widetilde{\fg}$ equipped with a proper surjective morphism $\pi:\widetilde{\fg}_1\to\fg_1$. We have a commutative diagram
\[
\begin{tikzcd}
 \widetilde{\fg}_1\ar[r, "\widetilde{\chi}_1"]\ar[d,"\pi"]&\fa\ar[d]\\
\fg_1\ar[r,"\chi_1"]&\fa/W_\fa, 
\end{tikzcd}
\]
 where $\fa$ is the universal Cartan of the symmetric pair, $\chi_1: \fg_1\to \fa/W_\fa$ is the categorical quotient map, and $\widetilde{\chi}_1$ is the restriction of $\widetilde{\chi}:\widetilde{\fg}\to \fg$ to $\fgbul$. Furthermore, the restriction 
 \[
 \widetilde{\chi}_1|_{\fgreg}:\fgreg:=\fgbul\times_{\fg_1}\fg_1^{reg}\lra \fa^{reg}
 \]is smooth, and the corresponding diagram is Cartesian.
\end{Thm}

See Section \ref{Section: relative Groth} for more details. The family of {quasi-split} symmetric pairs includes the ``diagonal'' symmetric space $(\fg_0\oplus\fg_0,\Delta\fg_0)$ as well as the stable (or split) involutions which feature in representation-theoretic approaches to arithmetic invariant theory (see \cite{ThorneVinberg}).  

\begin{Rem}

Our initial motivation for seeking such a result comes from considering the comparison of relative trace formulae. In many cases of interest (see \cite{LeslieFundamental}, for example), one needs to generalize results of Ng\^{o} on the Langlands-Shelstad fundamental lemma \cite{Ngo06} to the setting of symmetric spaces in order to stabilize these formulae. As noted above, the analogues of the results of Donagi and Gaitsgory we prove here will play a role for such generalizations.

Our reason for restricting to symmetric spaces is the current lack of a general theory of the fine structure of spherical varieties in positive characteristic. We nevertheless hope that our results here will aid in generalizations to spherical varieties. Our restriction to quasi-split involutions is also very natural from the perspective of harmonic analysis. For example, Prasad recently showed that generic representations over non-archimedean fields can be $G_0$-distinguished only for such involutions \cite{prasad2018generic}.
\end{Rem}
\begin{Rem}
 Despite the notation, $\fgbul$ is \emph{not} a simultaneous resolution of singularities of the categorical quotient $\fg_1\to \fg_1//G_0$. Even for the diagonal symmetric space, the space $\fgbul$ is not isomorphic to the Grothendieck-Springer resolution $\widetilde{\fg}_0$, though their pullbacks to the regular locus are obviously isomorphic. In Section \ref{Sec: Smoothness}, we identify a (Zariski-dense) interstitial space \[\fgreg\subset \fgres\subset \fgbul\] which is a family of resolutions of the singularities of $\fg_1\to \fg_1//G_0$. In particular, $\fgres$ is isomorphic to the Grothendieck-Springer resolution in the diagonal setting. We discuss this object in more detail toward the end of the introduction and in Section \ref{Sec: Smoothness}.
\end{Rem}

 The proof of Theorem \ref{Thm: main theorem intro} occupies Sections \ref{Section: relative Groth}, using several results from Sections \ref{Section: prelim} and \ref{Section: resolutions of nilp}. A key idea is to show (see Proposition \ref{Prop: canonical involution}) that the universal Cartan subspace $\ft$ of $\fg$ may be equipped with a canonical involution $\theta_{can}:\ft\to \ft$ associated to the symmetric pair. This allows us to identify the universal Cartan subspace $\fa$ of the symmetric pair $(\fg,\fg_0)$ as a distinguished subspace of $\ft$. The Grothendieck-Springer resolution is equipped with a smooth map $\widetilde{\chi}:\widetilde{\fg}\to \ft$, and we define
\begin{equation*}
\fgbul=\{(X,B)\in \fg_1\times \Fl_G:\widetilde{\chi}(X,B)\in \fa\},
\end{equation*}
and show that this space has all the desired properties. This relies on a classification of the irreducible components of the fiber product $\fg_1\times_{\ft/W}\ft$, which we characterize with the aid of a Kostant-Weierstrass section to the categorical quotient map $\chi_1:\fg_1\to \fg_1//G_0$ along with $G_0$-conjugacy results from \cite{Levy}. This section gives a distinguished component of the fiber product corresponding to $\fgbul.$

This component is intimately related to constructions of Knop in the context of spherical varieties \cite{knop1994asymptotic} in characteristic zero. We expound further on this relationship in Appendix \ref{Appendix: Knop}, showing how Knop's sections do not recover the Kostant-Weierstrass section in our setting. Our contribution is to show that Knop's construction works in positive characteristic for symmetric spaces, which is of interest to number theory. Additionally, our use of a Kostant-Weierstrass section allows for a much more explicit analysis of this space over the entire regular locus. Such an analysis is important for applications in number theory and symplectic geometry; we consider some applications below. 

For clarity, we state our explicit description of $\widetilde{\fg}^{reg}_1$ here:
\begin{equation}\label{eqn: proposal}
\widetilde{\fg}^{reg}_1:=\{(X,B)\in \fg^{reg}_1\times \mathcal{F}l_G : B(\theta) = Z_B(X_{ss}) \text{ is a regular $\theta$-stable Borel of }Z_G(X_{ss}) \}.
\end{equation}
To be more precise, we associate to an element $(X,B)\in \fg_1\times_{\fg}\widetilde{\fg}$ two subgroups of $B$. The first is the largest $\theta$-stable subgroup contained in $B$, given by $B(\theta)=B\cap \theta(B)$. For example, if $B$ is $\theta$-split, then $B(\theta)$ is a maximal torus. Second, if we denote by $X_{ss}$ the semi-simple part of $X$, then $X_{ss}\in \fg_1$ and the centralizer $Z_{B}(X_{ss})$ of $X_{ss}$ in $B$ is a Borel subgroup of the centralizer $Z_G(X_{ss})$. Finally, we define a $\theta$-stable Borel subgroup $B=\theta(B)$ to be \emph{regular} if its Lie algebra contains a regular nilpotent element $n\in \Lie(B)$ that lies in $\fg_1$. This notion arises naturally from studying the action $\theta$ induces on the Springer resolution of the nilpotent cone (see Section \ref{Section: resolutions of nilp}).

\subsection{Applications}
In proving Theorem \ref{Thm: main theorem intro}, we have two main applications in mind: the study of regular centralizers in $\fg_0$ for the action on $\fg_1$ (Section \ref{Section: regular stabilizers}) and potential applications to Springer theory for symmetric spaces (Section \ref{Sec: Smoothness}).

In Section \ref{Section: regular stabilizers}, we introduce the moduli space of regular stabilizers of the action of $G_0$ on $\fg_1$, denoted $\overline{G_0/N_0}$ where $N_0$ is the stabilizer in $G_0$ of a Cartan subspace $\fa$ of $\fg_1$. As the notation indicates, this space is a partial compactification of the space $G_0/N_0$ which parameterizes Cartan subspaces of $\fg_1$ \cite[Section 2]{Levy}. We show that this is naturally a smooth scheme. This space may be equipped with a natural $W_\fa$-cover $\overline{G_0/T_0}\to \overline{G_0/N_0}$, where $T_0$ is the centralizer of $\fa$ in $G_0$ and $W_\fa$ is the little Weyl group of the symmetric space. This cover is a partial compactification of the space $G_0/T_0$ of pairs $(\fa',\fb')$, with $\fa'\subset \fg_1$ a Cartan subspace of $\fg_1$ and $\fa'\subset \fb'$ where $\fb'$ is a $\theta$-split Borel subalgebra\footnote{Since we allow the characteristic to be positive, we work with the definition that a Borel subalgebra is simply the Lie algebra of a Borel subgroup.} of $\fg$ (Proposition \ref{Prop: split borels}). In Section \ref{Sec: regular stabilizers}, we prove that there is a Cartesian diagram
\[
\begin{tikzcd}
 \widetilde{\fg}_1^{reg}\ar[r]\ar[d]&\overline{G_0/T_0}\ar[d]\\
\fg_1^{reg}\ar[r]&\overline{G_0/N_0},
\end{tikzcd}
\]
and we show that the horizontal arrows in this diagram are smooth (see Proposition \ref{Prop: smoothness} and Theorem \ref{Thm: local etale}). A corollary of this is that the two $W_\fa$-covers  $\overline{G_0/T_0}\to \overline{G_0/N_0}$ and $\fa\to \fa/W_\fa$ are \'{e}tale-locally isomorphic in the strong sense that they become isomorphic after a smooth base change. This implies that one is \'{e}tale-locally a pull-back of the other and vice versa, whence the terminology. This is the analogue for quasi-split symmetric spaces of the results of \cite[Section 10]{DonGaits}. 

In Section \ref{Section: abelian groups}, we study the tautological sheaf of regular stabilizers
$$\mathcal{C}_0:=\{(g,\fc)\in G_0\times \overline{G_0/N_0}: g|_{\fc}=Id_\fc\}$$
 on $\overline{G_0/N_0}$. 
We prove that this group scheme is smooth and isomorphic to an abelian group scheme built out of the fixed point subgroup of the canonical involution on the universal Cartan $\theta_{can}:T\to T$. More precisely, let $T_0=T^{\theta_{can}}$ and consider the group scheme
\[
\calT_0(S):=\left(W_\fa\text{-equivariant morphisms }\widetilde{S}_0\to T_0\right)
\] 
for any $\overline{G_0/N_0}$-scheme $S$, where $\widetilde{S}_0=S\times_{\overline{G_0/N_0}}\overline{G_0/T_0}$. We show (see Theorem \ref{Thm: isomorphism}) that there is a canonical isomorphism $\mathcal{C}_0\iso \mathcal{T}_0$. Such a model for the sheaf of regular stabilizers is crucial for generalizing the approach of Ng\^{o} to studying fundamental lemmas in the context of relative trace formulae. 
\begin{Rem}
 While we assume for simplicity that $G_{der}$ is simply connected for much of the article, we address the necessary changes to obtain an isomorphism $\mathcal{C}_0\iso \calT_0$ in the general case in Section \ref{Section: not sc}.
\end{Rem}
\begin{Rem}
 This group scheme is intimately related to the automorphism group schemes used by Knop \cite{KnopAutomorphisms} in his analysis of collective invariant motion of a $G$-variety $X$ in characteristic zero. Recently, Sakellaridis \cite{Sakrank1} utilized Knop's group scheme in a crucial manner to prove a ``beyond endoscopic'' transfer statement for rank one spherical varieties. Interestingly, it is the ``complimentary subgroup'' $\mathcal{C}_0$ that is central to endoscopic phenomena in the symmetric case. 
\end{Rem}

Aside from motivations arising from relative trace formulae, we expect $\fgbul$ to have other applications in the representation theory of symmetric pairs. For example, Chen, Grinberg, Vilonen, and Xue (see \cite{chen2015springer,grinbergvilonenxue,vilonen2018character}) have recently studied analogues of Springer theory for symmetric pairs. While their initial work sought to generalize an approach of Lusztig which relies on $\widetilde{\fg}$,  their most general results rely on a near-by cycles construction in \cite{grinbergvilonenxue}. As noted above, the variety $\fgbul$ does not give a simultaneous resolution of singularities for the quotient $\fg_1\to \fg_1//G_0$, so it is natural to ask if there is an interstitial space 
\[
\fgreg\subset \fgres\subset \fgbul
\]
which generalizes the Grothendieck-Springer resolution in this sense. Toward this question, we consider in Section \ref{Sec: Smoothness} such a subspace $\fgres\subset \fgbul$, which recovers the classical Grothendieck-Springer resolution in the case of the case of the diagonal symmetric space $(\fg_0\oplus \fg_0,\Delta \fg_0)$. Our proposal for $\fgres$ is quite natural: we simply extend the construction of $\fgreg$ from (\ref{eqn: proposal}) to all of $\fg_1$. 
 \begin{Thm}Consider the subspace of $\widetilde{\fg}_1$ defined by
 \begin{equation*}
\widetilde{\fg}^{res}_1:=\{(X,B)\in \fg_1\times \mathcal{F}l_G : B(\theta) = Z_B(X_{ss}) \text{ is a regular $\theta$-stable Borel of }Z_G(X_{ss}) \},
\end{equation*}
and consider the map $\widetilde{\chi}_1: \widetilde{\fg}_1^{res}\lra \fa$. For each $a\in \fa$, there is a decomposition into connected components
\[
\widetilde{\chi}_1^{-1}(a)_{red}=\bigsqcup_{i\in \pi_0(\caln(a)_1)}\widetilde{\chi}_1^{-1}(a)_{w(i)}
\]
such that each component is smooth and the map $\widetilde{\chi}_1^{-1}(a)_{w(i)}\to {\chi}_1^{-1}(\overline{a})_{i}$ is a resolution of singularities. Here, $\widetilde{\chi}_1^{-1}(a)_{red}$ denotes the induced reduced scheme structure. 
 \end{Thm}
This is Theorem \ref{Thm: resolution fibers}, and give a sufficient criterion  in Lemma \ref{Lem: technical smooth} for this space to be smooth. Additionally, Proposition \ref{Prop: recovers groth} shows that this recovers the Grothendieck-Springer resolution as a special case. Thus, there is a precise way in which one may systematically delete $G_0$-orbits from $\fgbul$ to obtain a family of resolutions. As we note below, this family can fail to be smooth, or even irreducible, in general.

 Our argument is similar to the analysis of $\widetilde{\fg}$ in \cite[Chapter 3]{Slodowy}.
 In particular, we need a good understanding of the resolution of singularities of irreducible components of nilpotent cones of symmetric spaces. We review the construction and relevant properties of the resolution given by Sekiguchi and Reeder \cite{Sekiguchi, reeder1995} in Section \ref{Section: resolutions of nilp}, where we introduce the notion of a regular $\theta$-stable Borel subgroup and identify the subset of the fixed-point locus of the Springer resolution which arises in $\widetilde{\chi}^{-1}(0)$. 

However, there are very basic cases when the morphism $\chi_1:\fg_1\to \fg_1//G_0$ does not admit a simultaneous resolution. In such cases, our space $\fgres$ cannot be smooth and may not even give rise to an irreducible scheme. We describe a family of such examples using a monodromy argument in Section \ref{Sec: Smoothness}, but for a simple example, consider the case of a quasi-split symmetric pair $(\fsl(2),\fso(2))$. Then $\fg_1\cong \A_k^2$, $\fg_1//G_0\cong \A_k^1$, and these isomorphisms may be chosen so that $\chi_1$ corresponds to the map
\begin{align*}
\A_k^2&\lra\A_k^1\\ (x,y)&\longmapsto xy.
\end{align*}
In this case, only the fiber over $0\in \A^1$ is singular, given by two affine lines meeting transversely at one point. However, $\fg_1\times_{\fa/W_\fa}\fa$ is a cone, so that there is no way to resolve the singularity of $\fg_1\to \fg_1//G_0$ at $0$ while remaining birational to $\fg_1\times_{\fa/W_\fa}\fa$.  In this case, $\fgbul$ is the blow-up at the cone point and $\fgres=\fgbul\setminus{\Gm}$ where $\Gm=\SO(2)$ denotes the open $\SO(2)$-orbit of the exceptional fiber. The two remaining points of the exceptional fiber parameterize the two regular $\theta$-stable Borel subgroups of $\SL(2)$, or equivalently the two components of the nilpotent cone of $\fg_1$. 

This example illustrates both that $\fgbul$ is as close to the Grothendieck-Springer resolution for symmetric spaces as is possible in general and how one may obtain the resolution of singularities of fibers of $\fg_1\to \fg_1//G_0$ by systematically deleting $G_0$-orbits. In this sense, $\fgbul$ is the appropriate object to study in the case of symmetric spaces and we expect it to have applications to representation theory of the symmetric pair $(\fg,\fg_0)$ beyond those studied in the present article. We hope to study the connections between these spaces with the Springer theory developed in \cite{chen2015springer,grinbergvilonenxue,vilonen2018character} in future work.



Let us now summarize the paper. We review notation and certain basic properties of symmetric pairs in Section \ref{Section: prelim}. We then focus on quasi-split involutions, culminating in Proposition \ref{Prop: canonical involution}. In Section \ref{Section: resolutions of nilp}, we review the theory of the nilpotent cone $\caln_1\subset\fg_1$, studying the resolutions of the components of $\caln_1$. This will be used in the proof of Theorem \ref{Thm: resolution fibers}. We also introduce the notion of a regular $\theta$-stable Borel subgroup in this section. Section \ref{Section: relative Groth} introduces $\fgbul$, and proves Theorem \ref{Thm: main theorem intro}. In Section \ref{Section: regular stabilizers}, we turn to the primary application of studying the space of regular stabilizers $\overline{G_0/N_0}$ and the sheaf of regular stabilizers on this space. Finally, with an eye toward applications in Springer theory, we end by introducing the space $\fgres\subset \fgbul$ which is a (potentially non-smooth) family of resolutions of singularities of the quotient map. We give a criterion for when this space is smooth.

\subsection{Notation}

Algebraic groups will be denoted in Roman font, while Lie algebras will be in fraktur font. 

For any $G$-variety $V$ on which an endomorphism $\theta$ acts, we denote by $V^\theta$ the fixed point subvariety of $V$. For any subspace $U\subset \fg$, we denote its centralizer in a subgroup $H\subset G$ by $Z_H(U)$. In particular, for $X\in \fg_1$ we have
\[
Z_{G_0}(X) = Z_G(X)^\theta.
\]
We set $Z(G)$ to be the center of $G$. Similarly, we denote the centralizer of $U$ in the Lie algebra $\mathfrak{h}=\Lie(H)$ by $\fz_{\mathfrak{h}}(U)$. For any group $H$ on which $\theta$ acts, we denote $\tau(g)=g^{-1}\theta(g)$. 

For any group $H$, we use $H^\circ$ to denote the connected component of the identity. 


\subsection{Acknowledgements}
I want to thank Jayce Getz for introducing me to questions which led directly to this project, as well as for many helpful conversations. I also thank Ng\^{o} Bao Chau, Aaron Pollack, and David Treumann for helpful discussions. We thank the anonymous referee for several helpful suggestions. Finally, I want to thank Jack Thorne for comments that led to the discovery of an error in an earlier version of this article.

%
\setcounter{tocdepth}{1}%
\tableofcontents

\section{Preliminaries}\label{Section: prelim}

Let $k$, $G$, $\fg$, and $\theta$ be as above. We assume that $\mathrm{char}(k)\neq2$ is either $0$ or greater than $2\kappa$, where $\kappa$ is the supremum of the Coxeter numbers of the simple components of $G$. 
\begin{Rem}
Much of this article works for $\mathrm{char}(k)\neq2$ very good for $G$, which is a much weaker assumption. The only aspect relying on the restriction to $\mathrm{char}(k)>2\kappa$ is the theory of the resolutions of singularities of the nilpotent cone from \cite{reeder1995}. We expect that appropriate application of the techniques used in \cite{Levy} should allow for Reeder's results to be extended to good characteristic.
\end{Rem}
For simplicity, we assume that the derived subgroup $G^{(1)}$ of $G$ is simply connected, except in Section \ref{Section: not sc}. This is not a serious restriction since for any isogenous group $G'$ with involution $\theta'$ there exists a unique involution $\theta_{sc}$ of $G$ such that, if $p:G\to G'$ is the surjective isogeny, the diagram
\[
\begin{tikzcd}
G\ar[r,"\theta_{sc}"]\ar[d,"p"]&G\ar[d,"p"] \\
G'\ar[r,"\theta'"]&G'
\end{tikzcd}
\]
commutes; see \cite[9.16]{Steinberg} and \cite[Lemma 1.3]{Levy}. In particular, $\theta'$ and $\theta_{sc}$ induce the same involution on $\fg$. We abuse notation and also denote by $\theta: \fg\to \fg$ the associated linear involution of $\fg$. 

There is a direct-sum decomposition $\fg=\fg_0\oplus \fg_1$, where $\fg_i$ is the $(-1)^i$-eigenspace of $\theta$ in $\fg$. Let $G_0=\{g\in G: \theta(g)=g\}$ be the fixed point subgroup of $\theta$ in $G$. The assumption that $G_{der}$ is simply connected means that the connected components of $G_0$ is controlled by its image in the abelianization map $\nu:G/G_{der}\to G^{ab}\cong \Gm^k$. The restriction of the adjoint action to $G_0$ normalizes $\fg_1$, and $\fg_0=\Lie(G_0)$. We will often use $i\in\{0,1\}$ as a subscript to indicate objects associated to the corresponding $(-1)^i$-eigenspace; for example, we denote by $\caln_{1}$ the cone of nilpotent elements in $\fg_{1}$ (see Section \ref{Section: resolutions of nilp}). 

\subsection{Basics of symmetric pairs}
Let $(\fg,\fg_0)$ be a symmetric pair with associated involution $\theta$. We record here some structural facts about $(\fg,\fg_0)$ and point the reader to \cite{Levy} for more detail. We begin by noting that the Jordan decomposition behaves well with respected to the decomposition of $\fg=\fg_0\oplus\fg_1$.

\begin{Lem} For $X\in \fg$ and for $i=0,1$, $X\in \fg_{i}$ if and only if $X_{ss},X_{nil}\in \fg_{i}$ where $X=X_{ss}+X_{nil}$ is the Jordan decomposition of $X\in \fg$.
\end{Lem}

In particular, there is a well-defined notion of the semi-simple locus $\fg_1^{ss}$ of $\fg_1$, namely $\fg_1\cap\fg^{ss}$. 
 A toral subalgebra $\fa\subset \fg_1$ is a \emph{Cartan subspace} of $\fg_1$ if it is maximal in the collection of toral subalgebras of $\fg_1$. Such a subalgebra lies in the semi-simple locus of $\fg_1$. Define the rank of the symmetric space $r_1=\rk(\fg_1)$ to be $\dim(\fa)$ for a Cartan subspace $\fa$ (see \cite[Theorem 2.11]{Levy}).
A torus $A$ in $G$ is $\theta$-split if $\theta(a)=a^{-1}$ for all $a\in A$. A maximal such torus is called a maximal $\theta$-split torus. Any two maximal $\theta$-split tori of $G$ are conjugate by an element of $G_0$ \cite[Section 2]{Levy}.

We say an element $X\in \fg_1$ is regular if its centralizer $Z_{G_0}(X)\subset G_0$ has the smallest possible dimension, and denote $\fg_1^{reg}$ as the set of regular elements. We refer to \cite{KR71} for properties of regular elements. An element is regular semi-simple if it is both regular and semi-simple, and set $\fg^{rss}_1=\fg_1^{reg}\cap\fg_1^{ss}$ to be the regular semi-simple locus.

\subsection{Quasi-split symmetric pairs}
 Define a parabolic subgroup $P\subset G$ to be $\theta$-split if $P\cap \theta(P)$ is a Levi subgroup of $P$. Fix a maximal $\theta$-split torus $A$.

\begin{Prop}\cite[Section 1]{vust}\label{Prop: Vust split parabolics}
Let $P\supset A$ be a $\theta$-split parabolic subgroup. Then $P$ is minimal among $\theta$-split parabolic subgroup if and only if $P\cap \theta(P)=Z_G(A)$. Any two minimal $\theta$-split parabolic subgroups of $G$ are conjugate by an element of $G_0$.
\end{Prop}

\begin{Def}
A symmetric pair $(\fg,\fg_0)$ with associated involution $\theta$ is called \emph{quasi-split} if there exists a Borel subgroup $B$ that is $\theta$-split. This is equivalent to $B\cap \theta(B)$ being a torus. The pair $(\fg,\fg_0)$ (resp., $\theta$) is \emph{split} if it is quasi-split and the torus $B\cap\theta(B)$ is $\theta$-split.
\end{Def} 
We will be exclusively interested in quasi-split symmetric pairs in the sequel. The following characterizations are well known.

\begin{Prop}\label{Prop: quasisplit characterization}
A symmetric pair $(\fg,\fg_0)$ is quasi-split if and only if the following equivalent statements hold:
\begin{enumerate}
\item\label{item1} There exists a $\theta$-split Borel subgroup of $G$.
\item\label{item2} The centralizer of a maximal $\theta$-split torus is abelian.
\item\label{item3} There exists a regular element of $\fg$ contained in $\fg_1$; that is, $\fg_1\cap\fg^{reg}\neq\emptyset.$
\end{enumerate}
\end{Prop}
\begin{proof}
 Note that (\ref{item1}) and (\ref{item2}) are equivalent by Proposition \ref{Prop: Vust split parabolics}.
 
 Now let $\fa\subset \fg_1$ be a Cartan subspace of $\fg_1$ and let $A$ be the unique maximal $\theta$-split torus satisfying $\fa=\Lie(A)$ \cite[Lemma 2.4]{Levy}. Then the centralizing Levi subgroup $Z_G(\fa) = Z_G(A)$ is abelian if and only if it is a maximal torus of $G$. The equivalence between (\ref{item2}) and (\ref{item3}) now follows from Lemma 4.3 of \cite{Levy}, which implies that the dimension of $Z_G(x)$ is constant for all $x\in \fg_1^{reg}$.
 \end{proof}
We assume now and for the remainder of the paper that $(\fg,\fg_0)$ is quasi-split. Let $A\subset G$ be a maximal $\theta$-split torus. By Proposition \ref{Prop: quasisplit characterization}, $T:=Z_G(A)$ is a maximal torus. 

\subsection{The little Weyl group and $\theta$-split Borel subgroups}

Associated to the tori $A\subset T$, we have the absolute Weyl group $W_T$ and the little Weyl group $W_A=N_G(A)/Z_G(A)$. For a general symmetric pair, the little Weyl group $W_A$ is not naturally a subgroup of $W_T$, but a subquotient. 
When the symmetric pair is quasi-split, $W_A$ may be identified with the fixed-point subgroup $(W_T)^\theta:$

\begin{Lem}\label{Prop: weyl embedding}
When $\theta$ is quasi-split, there is a natural embedding
\[
W_A\hra W_T,
\]
where $T=Z_G(A)$ is the $\theta$-stable maximal torus containing $A$, where under this inclusion $W_A=(W_S)^\theta$.
\end{Lem}
\begin{proof}

By definition $W_A=N_G(A)/Z_G(A)$, and in this case $Z_G(A)=Z_G(T)=T$. This implies that $N_G(A)\subset N_G(T)$, giving the first claim.

Let $w\in (W_T)^\theta$ and suppose $n_w$ represents $w$. Then $\theta(n_w)=n_wt$ for some $t\in T$. We need to show that $n_w\in N_G(A)$. Indeed, for any $a\in A$, $n_wan_w^{-1}\in T$ and
\begin{align*}
\theta(n_wan_w^{-1})&=\theta(n_w)\theta(a)\theta(n_w)^{-1}\\
					&=n_w(ta^{-1}t^{-1})n_w^{-1}=n_wa^{-1}n_w^{-1}=(n_wan_w^{-1})^{-1},
\end{align*}
so that $n_wan_w^{-1}\in A$ giving the inclusion. Then the second claim now follows easily. 
\end{proof}

\begin{Rem}
 The above proposition gives an inclusion $W_A\subset W_T$ when $A$ is a maximal $\theta$-split torus and $T$ is its centralizer. If we instead consider a $\theta$-fixed Borel subgroup $B$ and $\theta$-stable maximal torus $T'\subset B$ and set $(W_{T'})_0=N_G(T')^\theta/Z_G(T')^\theta$, then we have the inclusions
\begin{equation}\label{eqn: weyl groups}
(W_{T'})_0\subset (W_{T'})^\theta\subset W_{T'}.
\end{equation}
The subscript $0$ is motivated by the fact that it is possible to choose $T'\subset B$ such that $T'_0:=(T'\cap G_0)^
\circ$ is a maximal torus in $G_0$ and $(W_{T'})_0=W(G_0,T'_0)$ is the Weyl group of $(G_0,T'_0)$. This distinction will be relevant in our discussion of resolutions of singularities of nilpotent cones in Section \ref{Section: resolutions of nilp}.
 \end{Rem}
\begin{Ex}
 Consider the simply connected form of $E_6$, and the following involution: let $\rho$ be the automorphism induced by the non-trivial diagram automorphism, and let $s=\check{\al}_0(-1)$, where $\al_0$ is the highest root, and $\check{\al}_0(t)$ is the corresponding cocharacter of $T$. Set $\theta=i_s\circ \rho$, where $i_s$ is conjugation by $s$. Setting $W=W_{T'}$, we have that $W^\theta$ is a Weyl group of type $F_4$. On the other hand, $W_0$ is the Weyl group of $G_0$ (which is type $C_4$). Thus, $[W^\theta: W_0]=3$, and $[W:W^\theta]=45$.

On the other hand, this corresponds to the  \emph{split} involution of type $E_6$ listed in \cite[pg. 549]{Levy}. It follows that $W_A=W_T$. \qed
\end{Ex}

Returning to our maximal $\theta$-split torus $A$ and centralizer $T$, note that there are $|W|$ Borel subgroups containing $A$. By \cite[Proposition 2.9]{springer85}, we know that there exists a $\theta$-split Borel subgroup $B\supset T$. The following proposition says that the $\theta$-split Borel subgroups containing $T$ is a $W_A$-torsor.

\begin{Prop}\label{Prop: split borels}
 Fix a $\theta$-split Borel $B\supset T$. Then any other $\theta$-split Borel $B'$ is of the form $wBw^{-1}$ for some $w\in W_A\subset W_T$. In particular, for any maximal $\theta$-split torus $A$, the set of $\theta$-split Borel subgroups containing it form a $W_A$-torsor.
\end{Prop}

\begin{Rem}
A slight variation of this argument shows that there is a $W_A$-torsor of minimal $\theta$-split parabolic subgroups $P$ containing a maximal $\theta$-split torus $A$ for arbitrary symmetric pairs. We leave the details to the reader.
\end{Rem}

\begin{proof}
Recall $W_A$ is the fixed-point subgroup of the induced action on $W=W_T$.  Any $w\in W^\theta$ takes $B$ to another $\theta$-split Borel subgroup. Indeed,
\begin{align*}
wBw^{-1}\cap\theta(wBw^{-1})&=wBw^{-1}\cap \theta(w)\theta(B)\theta(w)^{-1}\\
							&=wBw^{-1}\cap w\theta(B)w^{-1}\\ 
							&=w(B\cap\theta(B))w^{-1}=wTw^{-1}=T.
\end{align*}

To finish, for any other Borel $vBv^{-1}$ where $\theta(v)\neq v$, we claim that
\[
T\subsetneq vBv^{-1}\cap \theta(v)B^{op}\theta(v)^{-1}.
\]
Conjugating by $v$, the claim is equivalent to $T\subsetneq B\cap wB^{op}w^{-1}$ for some $w\neq1\in W_T$. This last claim is obvious by general theory, so we conclude that $vBv^{-1}$ is not $\theta$-split.
\end{proof}

\subsection{Canonical involution on the universal Cartan} \label{Section: quasisplit}

We end this section by recalling the universal Cartan subspace $\fa$ of a quasi-split symmetric pair $(\fg,\fg_0)$, and showing that the universal Cartan $\ft$ of $\fg$ inherits a canonical involution $\theta_{can}:\ft\to \ft$ such that $\fa$ may be identified as the $(-1)$-eigenspace. While we expect this is well known, we do not know of a reference for this result. We will make use of the induced embedding of universal Cartans $\fa\subset \ft$ in Section \ref{Section: relative Groth}.

\subsubsection{Canoncial Cartan of the symmetric variety} Let $X=G/H$ be a homogeneous variety of $G$ admitting an open orbit for some Borel subgroup $B$. Such varieties are called spherical, and symmetric varieties are special cases. To any such variety, one may attach a conjugacy of parabolic subgroups characterized as follows: let $B\subset G$ be a Borel subgroup, and let $\mathring{X}$ be the open $B$-orbit on $X$. We set $P(X)\supset B$ to be the maximal standard parabolic subgroup stabilizing $\mathring{X}$:
\[
P(X)=\{g\in G: g\mathring{X} =\mathring{X}\}.
\]
Define the \emph{universal Cartan subgroup of $G$} as the quotient $\mathcal{T}=B/[B,B]$. Note that for any other Borel subgroup $B'$, there is a canonical isomorphism
\[
\mathcal{T}=B/[B,B]\cong B'/[B',B'],
\]
justifying the name. This quotient inherits an action of the Weyl group $W$ of $G$, and the restriction of the quotient $B\to \mathcal{T}$ to any maximal torus $H\subset B$ induces a $W$-equivariant isomorphism $H\xrightarrow{\sim}\mathcal{T}$. We also have the Lie algebra version $\mathfrak{s}=\fb/[\fb,\fb]$; this is the universal Cartan subalgebra, which also inherits a $W$-action.

There is a canonical torus $\mathcal{A}_X$ associated to the variety $X$, known as the \emph{universal Cartan of $X$}. One may realize $\mathcal{A}_X$ as the quotient of $P(X)$ through which $P(X)$ acts on the quotient $U\backslash\backslash \mathring{X}$ where $U\subset B$ is the unipotent radical of $B$. In particular, we have quotient homomorphism of universal Cartans $\mathcal{T}\to \mathcal{A}_X$, and a corresponding map of Lie algebras $\ft\to \fa_X.$ Moreover, there is a finite Coxeter group $W_X$ associated to $X$, called the little Weyl group of $X$, which may be realized as a subquotient of $W$ and so that the quotient $\ft\to \fa_X$ is equivariant with respect to the appropriate subgroup of $W$.  The rank of $X$ is defined to be the rank of $\mathcal{A}_X$. 

While there are general issues with extending the theory of spherical varieties to fields of positive characteristic, this is not an issue in the case of a symmetric space $X^\theta=G/G_0$ (see \cite{Richardson} and \cite{Levy}), at least for large enough characteristics. We merely reference the terminology of general spherical varieties above for convenience. In this case, $P(X^\theta)$ is conjugate to a minimal $\theta$-split parabolic subgroup. In particular, $P(X^\theta)$ is a Borel subgroup under the quasi-split assumption. Moreover, $W_X$ is precisely the little Weyl group associated to a maximal $\theta$-split torus.
\begin{Lem}\label{Lem: canonical torus match}
For any maximal $\theta$-split torus $A$, there is an isogeny of tori $A\to\mathcal{A}_X.$
In particular, the two Lie algebras $\Lie(A)$ and $\fa_X$ are (non-canonically) isomorphic by an isomorphism which intertwines the actions of $W_A\cong W_X$.
\end{Lem}

\begin{proof}
Let $T=Z_G(A)$. For any $\theta$-split Borel subgroup $B\supset T$, consider the canonical isomorphism $\fb/[\fb,\fb]\cong \ft$. Restricting the quotient to $\Lie(T)$ induces an isomorphism $\Lie(T)\cong \ft$.

 The $\theta$-split condition implies that if $x=eG_0\in (X^\theta)^{open}$, then $B_x=B\cap G_0= B^\theta=T^\theta=T\cap G_0$. It follows from \cite[Lemma 1.3]{Levy} that $A\hra T\to T/T^\theta\cong U\backslash B/B_x\cong \mathcal{A}$ is an isogeny. Passing to Lie algebras gives the second statement, and the statement about Weyl group actions follows from the $W$-equivariance of $\ft\to \fa$.
\end{proof}

\subsubsection{The canonical involution}
 Hereafter, we will denote the universal Cartan of $G/G_0$ simply by $\mathcal{A}$, its Lie algebra by $\fa$, and the little Weyl group by $W_\fa$. Additionally, the notation $\ft$ will always denote the the Lie algebra of the universal Cartan of $G$ . The previous lemma and discussion imply that this is consistent with our previous notation, at least up to a non-canonical isomorphism. In particular, we have an inclusion of the Weyl groups $W_\fa\subset W$.

As is visible in the proof of the lemma, the isomorphism $\ft\cong \Lie(T)$ induced by any $\theta$-split Borel descends through the quotient $q:\Lie(T)\to\Lie(T)/\Lie((T\cap G_0)^\circ)\cong \Lie(A)$ to give a commutative diagram
\begin{equation}\label{eqn: diagram univ cartans}
\begin{tikzcd}
\ft\arrow[d,"\cong"] \arrow[r] &\fa \ar[d,"\cong"]\\
\Lie(T) \arrow[r,"q"] & \Lie(A). 
\end{tikzcd}
\end{equation}


 There is a natural splitting of $\Lie(T)\to \Lie(A)$ induced by the involution $\theta$ acting on $\Lie(T)$
\[
\Lie(T)=\Lie(A)\oplus\Lie(A)^\perp,
\]
where $\Lie(A)^\perp=\{X\in \Lie(T): \theta(X)=X\}$, and $q$ corresponds to the projection onto the first factor. The commutativity of the diagram induces a splitting $\fa\hra \ft$ for any choice of $\theta$-stable Borel subgroup. We claim that the image of this splitting is in fact independent of $A$, $T$, and $B$. 
\begin{Prop}\label{Prop: canonical involution}
There exists a canonical involution $\theta_{can}:\ft\to \ft$ inducing a decomposition
\[
\ft=\ft_0\oplus\ft_1.
\]
Moreover $\ft_1\cong \fa$ and the image of the splitting $\fa$ is $\ft_1$.
\end{Prop}
\begin{proof}
For any Borel subgroup $B$, there exists $gG_0\in G/G_0$ such that $B$ is $\theta^g$-split, where $\theta^g(h)=g\theta(g^{-1}hg)g^{-1}$ is the conjugate involution. Moreover, any other such involution is of the form $\theta^{bg}$ for some $b\in B$. Note that if $S=B\cap \theta^g(B)$ is the $\theta^g$-stable maximal torus of $B$ determined by $g$, then $bSb^{-1}=B\cap \theta^{bg}(B)$ for any $b\in B$. Thus for any pair $(B,S)$, there exists a conjugate involution $\theta^g$ such that $B$ is $\theta^g$-split and $S=B\cap \theta^g(B)$ is the distinguished $\theta^g$-stable Cartan subgroup.

Fix a Borel $B$ with an involution $\theta^g$ as above, and denote by $\theta^g$ be the induced involution on $\Lie(S)$. We have the induced isomorphism
\[
\varphi_B:\Lie(S)\hra \fb\to \fb/[\fb,\fb]\cong \ft,
\]
and consider the involution $\theta^\ast$ on $\ft$ induced by this isomorphism. For any other Borel subgroup $P$ and involution $\theta^h$ such that $P$ is $\theta^h$-split with corresponding stable torus $T$, there is a $g_1\in G$ such that $(P,T)=(g_1Bg_1^{-1},g_1Sg_1^{-1})$. 
The choice of $g_1$ is determined up to the $T\times S$-action $(t,s)\cdot g_1\mapsto tg_1s$, so that the induced map $\ad(g_1):\Lie(S)\xrightarrow{\sim}\Lie(T)$ is independent of all choices and is equivariant with respect to the involutions:
\[
\theta^h(\ad(g_1)(X))=\ad(g_1)(\theta^{g}(X)).
\]

Let $\theta^{\ast\ast}$ denote the involution on $\ft$ induced by $\varphi_P:\Lie(T)\to\ft$. We have the commutative diagram
\[
\begin{tikzcd}
\Lie(S)\arrow[d,"\varphi_B"] \arrow[r, "\ad(g_1)"] &\Lie(T) \ar[d,"\varphi_P"]\\
\ft \arrow[r,"="] & \ft. 
\end{tikzcd}
\]
Indeed, the unique isomorphism $\fb/[\fb,\fb]\cong\mathfrak{p}/[\mathfrak{p},\mathfrak{p}]$ is induced by $\ad(g_1)$. Note that for $t\in \ft$ there exists a unique $X\in \Lie(S)$ such that $t\cong X\pmod{[\fb,\fb]}$ and a unique $Y\in \Lie(T)$ such that $t\cong Y\pmod{[\fp,\fp]}$. Then the above diagram implies $\ad(g_1)(X)=Y$, so that
\begin{align*}
\theta^{\ast\ast}(t)&\cong \theta^h(Y)\qquad\quad\:\;\pmod{[\fp,\fp]}\\
		&=\ad(g_1)(\theta^{g}(X))\pmod{[\fp,\fp]}\\
		&\cong \theta^{g}(X)\qquad\quad\:\;\pmod{[\fb,\fb]}\cong \theta^\ast(t),
\end{align*}
where every isomorphism used is the canonical one. In particular, the two involutions are identified. We conclude that induced involution $\theta_{can}:\ft\to \ft$ is independent of the choices involved. Let $\ft=\ft_0\oplus \ft_1$ be the induced decomposition, where $\ft_i$ is the $(-1)^i$-eigenspace of $\theta_{can}$.

Now consider the case of a $\theta$-split Borel $B$ with maximal $\theta$-split torus $A\subset S=B\cap \theta(B)$. Then the construction of $\theta_{can}$ implies we have an $(\theta,\theta_{can})$-equivariant isomorphism
\[
\varphi_B:\Lie(S)\xrightarrow{\sim}\ft.
\]
In particular, we obtain an isomorphism $\Lie(A)\cong \ft_1$ of $(-1)$-eigenspaces. By the commutative diagram (\ref{eqn: diagram univ cartans}), it follows that the section $\fa\hra \ft$ is an isomorphism onto $\ft_1$.
\end{proof}

We remark that a similar argument produces a canonical involution $\theta_{can}:\mathcal{T}\to\mathcal{T}$ which differentiates to the involution discussed in the proposition. 

\begin{Cor}\label{Cor: universal torus}
Let $\mathcal{T}$ be the universal Cartan of $G$. There is a canonical involution $\theta_{can}:\mathcal{T}\to\mathcal{T}$. In particular, there is a universal regular fixed-point subgroup $\mathcal{T}_0=\mathcal{T}^{\theta_{can}}$. 
\end{Cor}
We will use this corollary in Section \ref{Section: regular stabilizers} when the universal fixed-point torus $\mathcal{T}_0$ is used to study the universal stabilizer group scheme.

\section{Nilpotent cones of symmetric spaces}\label{Section: resolutions of nilp}

In this section, we discuss the nilpotent cone $\caln_{1}=\caln\cap\fg_1$ and desingularizations of nilpotent $G_0$-orbits. We introduce the notion of a regular $\theta$-stable Borel subgroup for use in Section \ref{Section: relative Groth}. Aside from this definition, this section will be used in Section \ref{Sec: Smoothness} to study the generalization of the Grothendieck-Springer resolution over the entire space $\fg_1$.

 The variety $\caln_1$ need not be irreducible.
In fact, there is a bijection between connected components of $\caln_1^{reg}= \caln_1\cap \fg_1^{reg}$ and irreducible components of $\caln_1$. Motivated by this, we adopt the notation $\pi_0(\caln_1)$ to denote the set of irreducible components of $\caln_1$. We refer the reader to \cite{KR71} in characteristic zero and \cite{Levy} in good characteristic for further details. 
 There is a general construction of resolutions of singularities for nilpotent orbit closures due to \cite{reeder1995, Sekiguchi} which generalizes the Springer resolution of the nilpotent cone. As we are working in the special case of quasi-split symmetric spaces and only consider resolutions of regular nilpotent orbits, we describe the resolution in a simpler, albeit less general fashion.

Fix a \emph{regular nilpotent} $e\in\caln_1$ which lies in the Lie algebra $\Lie(B)$ of a unique Borel subgroup $B$, which is necessarily $\theta$-stable. Recall the Springer resolution of the nilpotent cone of $\fg$:
\[
\widetilde{\caln}=\{(X,B)\in \caln\times \mathcal{F}l_G: X\in \Lie(B)\}\cong G\times^B\fn,
\]
where we may choose $B$ to be our $\theta$-stable Borel subgroup. The map $\check{\pi}:\widetilde{\caln}\lra \caln$ given by $(X,B)\mapsto X$ is the Springer resolution of singularities. Consider the involution $\theta^\ast$ on $\widetilde{\caln}$ defined by
\[
\theta^\ast(X,B)= (-\theta(X),\theta(B)).
\]

 Fixing a $\theta$-stable torus $T\subset B$, we denote for the remainder of this section $W=W_T$. In the previous section (\ref{eqn: weyl groups}), we introduced two subgroups $W_0\subset W^\theta\subset W$. Reeder shows in \cite{reeder1995} that the fixed-point variety $\widetilde{\caln}^{\theta^\ast}$ may be decomposed as a disjoint union of vector bundles over $\Fl_{G_0}$ indexed by $W_0\backslash W^\theta$:
\[
\widetilde{\caln}^{\theta^\ast}=\bigsqcup_{w\in W_0\backslash W^\theta}E_w.
\]
The restriction of $\check{\pi}$ naturally maps to $\caln_1$, and we have the following:
\begin{Prop}\label{Prop: nilp fixed points}
Assume that $\theta$ is quasi-split. Then for each component $\caln_1^i$, there exists exactly one $w=w(i)\in W_0\backslash W^\theta$ such that the restriction of $\check{\pi}$ to $E_{w(i)}$ is a resolution of singularities
\[
\check{\pi}: E_{w(i)}\lra \caln_1^i.
\]
\end{Prop}
\begin{proof}
This follows from \cite[Proposition 3.2]{reeder1995}, the proof of \cite[Proposition 4.1]{reeder1995}, and our assumption that $(\fg,\fg_0)$ is quasi-split.
\end{proof}
In general, the number $[W^\theta:W_0]$ is greater than $\#\pi_0(\caln_1)$. In particular, there may exist $\theta$-stable Borel subgroups $B\subset G$ such that $\Lie(B)\cap\caln_1^{reg}=\emptyset$. 
\begin{Ex}
Consider the split involution for the exceptional Lie algebra $\fg_2$. Then $\fg_0\cong \fsl(2)\oplus \fsl(2)$ and $\fg_1\cong V\boxtimes\mathrm{Sym}^3(V)$, where $V$ is the standard representation. In this case, the nilpotent cone $\caln_1$ is irreducible \cite[Lemma 6.19 (c)]{Levy}, but $[W^\theta:W_0]=3$ since this is an inner involution. Thus, there is only a single orbit of $\theta$-stable Borel subgroups meeting $\caln^{reg}_1$, and there are two orbits which do not.
\end{Ex}
\begin{Def}
 Suppose that $B\in\Fl_{G}^\theta$ is a $\theta$-stable Borel subgroup.  If this intersection $\Lie(B)\cap \caln_1^{reg}$ is non-empty, we say that $B$ is a \textbf{regular} $\theta$-stable Borel subgroup. 
\end{Def}
It is only regular $\theta$-stable Borel subgroups that contribute to the fibers of the resolutions in Proposition \ref{Prop: nilp fixed points}. 
Denote the set of \emph{regular} $\theta$-stable Borel subgroup of $G$  by $(\Fl_{G}^\theta)^{reg}$, so that
\[
(\Fl_{G}^\theta)^{reg}=\bigsqcup_{i\in \pi_0(\caln_1)}C_{w(i)},
\]
where $C_{w(i)}=\{B\in\Fl_{G}: \Lie(B)\cap (\caln_1^i)^{reg}\neq \emptyset\}$ is the $G_0$-orbit of regular $\theta$-stable Borel subgroups whose Lie algebras meet the regular locus of the component $\caln_1^i\subset\caln_1$. Note that each $C_{w(i)}$ is a closed $G_0$-orbit by Proposition 2.3 of \cite{reeder1995}.
\section{A simultaneous resolution over the regular locus}\label{Section: relative Groth}


In this section, we define and study a subscheme $\fgbul\subset \fg_1\times_\fg\widetilde{\fg}$ which fits into a diagram analogous to the Grothendieck-Springer resolution where $\ft$ is replaced by the universal Cartan subspace of $\fg_1$ and prove Theorem \ref{Thm: main theorem intro}, which we now recall.
\begin{Thm}\label{Thm: main theorem}
Let $(\fg,\fg_0)$ be a quasi-split symmetric pair with $\fg=\fg_0\oplus \fg_1$. There is a subscheme $\fgbul\subset \fg_1\times_\fg\widetilde{\fg}$ with a proper surjective morphism $\pi:\widetilde{\fg}_1\to\fg_1$. We have a commutative diagram
\begin{equation*}
\begin{tikzcd}
 \widetilde{\fg}_1\ar[r, "\widetilde{\chi}_1"]\ar[d,"\pi"]&\fa\ar[d]\\
\fg_1\ar[r,"\chi_1"]&\fa/W_\fa, 
\end{tikzcd}
\end{equation*}
 where $\fa$ is the universal Cartan of the symmetric pair, $\chi_1: \fg_1\to \fa/W_\fa$ is the categorical quotient map, and $\widetilde{\chi}_1$ is the restriction of $\widetilde{\chi}:\widetilde{\fg}\to \fg$ to $\fgbul$. Furthermore, the restriction 
 \[
 \widetilde{\chi}_1|_{\fgreg}:\fgreg:=\fgbul\times_{\fg_1}\fg_1^{reg}\lra \fa^{reg}
 \]is smooth, and the corresponding diagram is Cartesian.
\end{Thm}

We prove this theorem in the  next section by defining $\widetilde{\fg}_1$ to be a distinguished irreducible component of the fiber product $\fg_1\times_\fg\widetilde{\fg}$, proving several desirable properties including the statement about the restriction to the regular locus.



\subsection{Components of the fiber product}\label{subsection: component}

Consider the Cartesian diagram
\[
\begin{tikzcd}
 {\fg}_1\times_{\ft/W}\ft\ar[d] \arrow[r] & \ft\ar[d,"\pi"]\\
\fg_1 \arrow[r,"\chi"] &\ft/W. 
\end{tikzcd}
\]
The fiber product is not irreducible, and we must study the various irreducible components.
\begin{Prop}\label{Prop: components}
The irreducible components of ${\fg}_1\times_{\ft/W}\ft$ all have the same dimension. They each surject onto $\fg_1$, and are permuted transitively by the Weyl group action on the second factor. Finally, each component is stable under the $G^\circ_0$-action on the left.
\end{Prop}
This is an infinitesimal version of \cite[Lemma 6.5]{Knop95}, and the proof is essentially the same. Indeed, much of this subsection is analogous to the arguments in \cite[Section 3]{knop1994asymptotic}. See Appendix \ref{Appendix: Knop} for a discussion on the relations between this section and the work of Knop.
\begin{proof}
We claim that ${\fg}_1\times_{\ft/W}\ft$ is a complete intersection in $\fg_1\times \ft$. To see this, note that $\ft/W\iso \A_k^{r_1}$ is an affine space of dimension $r_1=\dim(\fa)$ and the morphism $\ft\to \ft/W$ is flat of relative dimension $0$ with $\ft$ smooth. This implies that ${\fg}_1\times_{\ft/W}\ft\to \fg_1$ is also flat of relative dimension $0$, so that $$\dim({\fg}_1\times_{\ft/W}\ft) = \dim({\fg}_1\times\ft)-r_1.$$ 
The inclusion $\fa\subset \ft$ and the commutative diagram
\[
\begin{tikzcd}
\fg_1\ar[d,"\chi_1"] \arrow[r] & \fg\ar[d,"\chi"]\\
\fa/W_\fa \arrow[r] &\ft/W. 
\end{tikzcd}
\]
implies that ${\fg}_1\times_{\ft/W}\ft\subset \fg_1\times \fa$. More precisely, it is the zero set of the $r_1$ equations induced by the coordinates of $\chi_1(g)= \pi(t)$. Thus, ${\fg}_1\times_{\ft/W}\ft$ is a complete intersection in $\fg_1\times \ft$.

This implies that all the components have the same dimension. Since the fibers of ${\fg}_1\times_{\ft/W}\ft\to \fg_1$ are $W$-orbits, each component maps finitely-to-one onto $\fg_1$, and $W$ acts transitively on the components. 
The final statement follows from the fact that $G_0^\circ$ is connected, and that the fibers of ${\fg}_1\times_{\ft/W}\ft\to \ft/W$ are $G_0$-stable.
\end{proof}

We now make these components more explicit. Set $\hat{\fg}_1:=\fg_1\times_{\fa/W_\fa}\fa$, so that there is a Cartesian diagram

\begin{equation}\label{eqn: diagram 1}
\begin{tikzcd}
 \hat{\fg}_1\ar[d,"\pi"] \arrow[r,"\hat{\chi}_1"] & \fa\ar[d]\\
\fg_1 \arrow[r,"\chi_1"] &\fa/W_\fa. 
\end{tikzcd}\qedhere
\end{equation}


Fix a set of coset representatives $v\in W/W_\fa$. Then for each $v\in W/W_\fa$, define ${v}:\fa\to \ft$ by ${v}(a)=v\cdot a$. Then the composition 
\[
\fa\xrightarrow{{v}}\fa\to \ft/W
\]
is $W_\fa$-invariant, so that it factors uniquely to give a diagram
\begin{equation}\label{eqn: diagram 2}
\begin{tikzcd}
 \fa\ar[d,] \arrow[r,"{v}"] & \ft\ar[d]\\
\fa/W_\fa \arrow[r,"{v}/W_A"] &\ft/W,
\end{tikzcd}
\end{equation}
by the universal property of the categorical quotient. Composing (\ref{eqn: diagram 1}) with (\ref{eqn: diagram 2}), we obtain a closed embedding, also denoted by ${v}$,
\[
{v}:\hat{\fg}_1\to \fg_1\times_{\ft/W}\ft.
\]
Denoting the image by $C_v\subset \fg_1\times_{\ft/W}\ft$, then
 $C_v\to \fg_1$ is surjective for each $v$.
\begin{Lem}\label{Lem: irreducible}
$C_v$ is irreducible for each $v\in W/W_\fa$. 
\end{Lem}
\begin{proof}

It suffices to prove $\hat{\fg}_1$ is irreducible. 
 Recalling that since $(\fg,\fg_0)$ is quasisplit, the intersection of $\fg_1$ with the regular semi-simple locus of $\fg$ is non-empty (if fact, it is dense). Set $\fg_1^{rss}=\fg_1\cap\fg^{rss}$. Since $W_\fa$ permutes the irreducible components, it suffices to show that $\hat{\fg}_1^{rss}$ is irreducible. This will follow from the existence and properties of the Kostant-Weierstrass section, as we now explain.

Fix a regular nilpotent element $e\in \fg_1$. Then there exists an $r_1=\rk(\fg_1)$ dimensional affine subspace $e+\mathfrak{v}\subset \fg_1$ such that (see \cite[Section II.3]{KR71} for characteristic zero and \cite[Lemma 6.30]{Levy} for good characteristics):
\begin{enumerate}
\item The restriction $\chi_1|_{e+\mathfrak{v}}:e+\mathfrak{v}\lra  \fa/W_\fa$ is an isomorphism, 
\item every element $X\in e+\mathfrak{v}$ is regular in $\fg_1$, and
\item each regular $G_0^\ast$-orbit in $\fg_1$ meets $e+\mathfrak{v}$ in exactly one point.
\end{enumerate}
Here, $G_0^\ast=\{g\in G: g^{-1}\theta(g)\in Z(G)\}=FG_0^\circ$, where $F=\{a\in A: a^2\in Z(G)\}$. Let $\kappa$ denote the inverse isomorphism $\kappa:\fa/W_\fa\to e+\mathfrak{v}$, known as a Kostant-Weierstrass section. Consider the morphism $\ka_{KW}: \fa\to \hat{\fg}_1$ defined by
\[
\ka_{KW}(a)=(\kappa(\overline{a}),a).
\]
This is a section of $\hat{\chi}_1:\hat{\fg}_1\to \fa$, so that the image is an irreducible closed subscheme of $\hat{\fg}_1$ with an open dense subscheme $\ka_{KW}(\fa^{reg})$. This implies that $G^\circ_0\cdot\ka_{KW}(\fa^{reg})$ is irreducible, as $G^\circ_0$ is smooth and connected. Applying \cite[Lemma 6.29]{Levy}, we see that
\[
G^\circ_0\cdot\sigma(\fa^{reg})=\hat{\fg}_1^{rss},
\]
 implying that $\hat{\fg}_1$ is irreducible.
\end{proof}


Let $\mathcal{I}$ denote the set of irreducible components of ${\fg}_1\times_{\ft/W}\ft$.
\begin{Cor}\label{Lem: enumerate components}
The map $v\mapsto C_v$ is a bijection between 
\[
W/W_\fa\lra \mathcal{I}.
\]
\end{Cor}
\begin{proof}
By the previous lemma, the map is well defined. 
 Noting that
\[
{\fg}^{rss}_1\times_{\ft/W}\ft=\bigcup_vC_v^{rss},
\]
and $C_v^{rss}\cap C_w^{rss}=\emptyset$ if $v\neq w\in W/W_A$, the corollary now follows.
\end{proof}

\subsection{Over the regular locus}

Set $C_1^{reg}$ for the restriction of $C_1$ to the regular locus. Since $C_1$ is isomorphic to the fiber product $\fg_1\times_{\fa/W_\fa}\fa$, we have a Cartesian diagram
\begin{equation}\label{eqn: cartesian1}
\begin{tikzcd}
 C_1^{reg}\ar[r, "\widetilde{\chi}_1"]\ar[d,"\pi"]&\fa\ar[d]\\
\fg^{reg}_1\ar[r,"\chi_1"]&\fa/W_\fa.
\end{tikzcd}
\end{equation}
For our applications, we need another description of $C_1^{reg}$. There is a natural proper map ${\fg}_1\times_{\fg}\widetilde{\fg}\to {\fg}_1\times_{\ft/W}\ft$ induced by the map $\widetilde{\fg}\to \fg\times_{\ft/W}\ft$. Moreover, if we restrict to the regular locus, we obtain an isomorphism
\[
{\fg}^{reg}_1\times_{\fg}\widetilde{\fg}\xrightarrow{\sim} {\fg}^{reg}_1\times_{\ft/W}\ft,
\]
where we use the fact that $\fg_1^{reg}\subset \fg^{reg}$ and that $\fg_1^{reg}\times_{\fg^{reg}}\left(\fg^{reg}\times_{\ft/W}\ft\right)\cong \fg^{reg}_1\times_{\ft/W}\ft$. Corollary \ref{Lem: enumerate components} thus enumerates those irreducible components of ${\fg}_1\times_{\fg}\widetilde{\fg}$ that map onto the regular locus of $\fg_1$. 

In particular, there is a unique irreducible component
\[
\fgbul\subset{\fg}_1\times_{\fg}\widetilde{\fg}
\]such that $\widetilde{\fg}_1^{reg}\xrightarrow{\sim}C^{reg}_1$. Setting $\pi:\fgbul\to \fg_1$ for the induced proper morphism, we obtain a commutative diagram
 \[
 \begin{tikzcd}
 \widetilde{\fg}_1\ar[r, "\widetilde{\chi}_1"]\ar[d,"\pi"]&\fa\ar[d]\\
\fg_1\ar[r,"\chi_1"]&\fa/W_\fa, 
\end{tikzcd}
 \] By our previous considerations, this diagram is Cartesian over the regular locus of $\fg_1$.  In particular, \cite[Corollary 6.31]{Levy} implies that $\widetilde{\chi}_1$ is smooth over the regular locus. 
 


 We now give a description of this scheme in terms of Borel subgroups of $G$. For an element $(X,B)\in \fg_1\times_{\fg}\widetilde{\fg}$, we define the following two subgroups. Firstly, let $$B(\theta)=B\cap \theta(B)$$ denote the largest $\theta$-stable subgroup of $B$; it has the Lie algebra $\fb(\theta)=\fb\cap\theta(\fb)$. Secondly, let $$Z_B(X_{ss})=B\cap Z_G(X_{ss})$$ be the corresponding Borel subgroup of $Z_{G}(X_{ss})$, where $X=X_{ss}+X_{nil}$ is the Jordan decomposition. Denote by $\fz_{\fb}(X_{ss})$ the Lie algebra of $Z_B(X_{ss})$. 

\begin{Prop}\label{Prop: regular characterization}
With the definitions as above, we have that
\begin{equation*}
\widetilde{\fg}^{reg}_1=\{(X,B)\in \fg^{reg}_1\times \mathcal{F}l_G : B(\theta) = Z_B(X_{ss}) \text{ is a regular $\theta$-stable Borel of }Z_G(X_{ss}) \}.
\end{equation*}
\end{Prop}
This proposition proves Theorem \ref{Thm: main theorem}. For later reference, we refer to such Borel subgroups as maximally split regular Borel subgroups. This terminology is justified by the observation that for any $X\in \fg_1^{reg}$ and any Borel subgroup $B$ with $X\in \Lie(B)$, we have the inclusion $Z_B(X_{ss})\subset B(\theta)$. 
\begin{proof}
Let $S\subset \fg_1^{reg}\times_{\fg_1}\widetilde{\fg}$ denote the right-hand side. We first show that the map 
\[
S\lra \fg_1\times_{\ft/W}\ft\lra  \ft
\]factors through $\fa$. For this we make use of the canonical involution $\theta_{can}:\ft\lra \ft$ from Proposition \ref{Prop: canonical involution} and the fact that $\fa\subset \ft$ is the $(-1)$-eigenspace for this involution.

Let $g\in Z_G(X_{ss})$ be such that $g^{-1}B(\theta) g$ is split for the restriction of $\theta$ to $Z_G(X_{ss})$. Note that
 $$X\equiv X_{ss}\pmod{[\fb,\fb]},$$ 
so we are free to assume $X=X_{ss}$. Note that $Z_G(X_{ss})B=P(X_{ss})$ is a parabolic subgroup of $G$ with Levi subgroup $Z_G(X_{ss})$. If $P(X_{ss}) =Z_G(X_{ss})U^P$ where $U^P$ denotes the unipotent radical of $P(X_{ss})$, set $U^\theta=B\cap U^P$. Then $U^\theta$ is the largest unipotent subgroup of $B$ such that $\theta(U^\theta)\cap U^\theta=1$, and we have the decomposition $B=B(\theta)\cdot U^\theta$.  In particular, for $g\in Z_{G}(X_{ss})$ we have $g^{-1}U^\theta g\subset U^\theta$. We claim that $g^{-1}Bg$ is $\theta$-split. Indeed,
\[
\theta(g^{-1}Bg)=\theta(g^{-1}B(\theta) g)\theta(U^\theta).
\]
Noting that $g^{-1}B(\theta)g$ is a Borel subgroup of $Z_{G}(X_{ss})$, the Levi decomposition for $P(X_{ss})$ implies $$\theta(g^{-1}Bg)\cap g^{-1}Bg = \theta(g^{-1}B(\theta) g)\cap g^{-1}B(\theta) g$$ is a maximal torus in $Z_G(X_{ss})$. Thus, $B$ is $\theta^g$-split.

Since $g\in Z_G(X_{ss})$, 
\[\theta^g(X_{ss})=\Ad(g^{-1})\circ\theta\circ\Ad(g)(X_{ss})=-X_{ss},\] and it follows by the definition of $\theta_{can}$ that
\begin{align*}
\theta_{can}\left[X_{ss}\pmod{[\fb,\fb]}\right] = \theta^g(X_{ss})\pmod{[\fb,\fb]}=-X_{ss}\pmod{[\fb,\fb]}.
\end{align*}
Thus, the map $S \to \ft$ factors through $\fa$, implying that we have a map $S\to \fgreg$. Since $W_\fa$ acts transitively on the fibers of $\fgreg$, the argument above and Proposition \ref{Prop: split borels} combine to show that this map is an isomorphism on geometric points. As $\fgreg$ is smooth, this is sufficient.
\end{proof}

We explicate the fibers of $\pi_1:\fgreg\to \fg^{reg}_1$ on geometric points. Suppose that $X=X_{ss}+X_{nil}\in \fg_1^{reg}$, and let $\fa$ be a Cartan subspace of $\fg_1$ containing $X_{ss}$. Then $A\subset Z_G(X_{ss})$, where $\Lie(A)=\fa$. Let $B_{split}$ be a $\theta$-split Borel subgroup containing $A$. Then $P(X)=Z_{G}(X_{ss})B_{split}$ is a $\theta$-split parabolic subgroup with Levi subgroup $Z_G(X_{ss})$. The assumption that $X$ is regular is equivalent to $X_{nil}\in \fz_\fg(X_{ss})^{reg}$ \cite[Theorem 7]{KR71}. Therefore, there is a unique Borel subgroup $B\subset Z_{G}(X_{ss})$ such that $X_{nil}$ lies in the nilradical of $\Lie(B)$. Setting $B'=BU_P$, where $U_P$ is the unipotent radical of $P(X)=Z_G(X_{ss})U_P$, then $B'$ is a Borel subgroup of $G$ and $(X,B')\in \pi^{-1}(X)$. This sets up a bijection
\begin{align*}
\{\text{$\theta$-split parabolic subgroups with Levi }Z_G(X_{ss})\}&\longleftrightarrow \pi^{-1}(X)\\
							P(X)=Z_{G}(X_{ss})B_{split}\qquad\qquad	&\longleftrightarrow \:(X,B')
\end{align*}
Since any two $\theta$-split Borel subgroups $B_1,B_2\supset A$ give the same parabolic subgroup $P(X)$ if and only if $B_1=wB_2w^{-1}$ for some $w\in \Stab_{W_\fa}(X_{ss})$, the left-hand side is in bijection with $W_\fa/\Stab_{W_\fa}(X_{ss})\cong W_\fa\cdot X_{ss}$. Thus, this gives the entire fiber.

Noting that since $W_\fa$ permutes the Borel subgroups in the fiber over a given regular element $X\in \fg_1^{reg}$ through the action of $W_A$ for some maximal $\theta$-split torus $A$ contained in $Z_G(X_{ss})$ and $W_A=N_{G_0}(A)/Z_G(A)$, we have the following corollary.
\begin{Cor}\label{Cor:regular Borels orbit}
For a regular element $X\in \fg_1^{reg}$, $\pi_2(\pi^{-1}(X))\subset \Fl_G$ lies in a single $G_0$-orbit, where $\pi_2:\widetilde{\fg}_1\to \Fl_G$ is the natural map. Furthermore, for any $X\in \fg_1$ if $(X,B_1),(X,B_2)\in \pi^{-1}(X)$ and $B_1(\theta)=B_2(\theta)$, then $B_1$ is $G_0$-conjugate to $B_2$.
\end{Cor}
\begin{proof}
The first claim follows from the discussion above. For the second claim, we first assume that $X\in \fg_1^{ss}$. Fixing $e\in \Lie(B_1(\theta))\cap \caln(\fz_\fg(X))_1^{reg}$, then $(X+e,B_i)\in \widetilde{\fg}^{reg}_1$ for $i=1,2$ and $X+e\in \fg_1^{reg}$. The second claim now follows from the first claim for $X\in \fg_1^{ss}$. For general $X\in\fg_1$, note that $(X,B_1),(X,B_2)\in \pi^{-1}(X)$ implies that $(X_{ss},B_1),(X_{ss},B_2)\in \pi^{-1}(X_{ss})$, where $X=X_{ss}+X_{nil}$ is the Jordan decomposition of $X$. 
\end{proof}
\begin{Rem}
It is natural to ask for an explicit description of $\fgbul$. By the construction of $C_1$, we have that $(X,B)\in\fgbul$ if and only if $X\pmod{[\fb,\fb]}\in \fa\subset \ft$. For example, $(0,B)\in \fgbul$ for any $B\in \Fl_G$. Comparing with Proposition \ref{Prop: nilp fixed points}, we conclude the diagram 
\[
\begin{tikzcd}
\fgbul\ar[r, "\widetilde{\chi}_1"]\ar[d,"\pi"]&\fa\ar[d]\\
\fg_1\ar[r,"\chi_1"]&\fa/W_\fa, 
\end{tikzcd}
\]
does not give a simultaneous resolution of singularities and the map $\fgbul \to \fg_1$ is not small. We discuss the question of whether there is an interstitial space $\fgreg\subset \fgres\subset\fgbul$ in Section \ref{Sec: Smoothness} below.
\end{Rem}


\section{Moduli space of regular stabilizers}\label{Section: regular stabilizers}

In this section, we generalize to the case of quasi-split symmetric spaces several results of Donagi and Gaitsgory \cite[Section 10]{DonGaits}. These fundamentally rely on Theorem \ref{Thm: main theorem} over the regular locus.

\subsection{Regular stabilizers}\label{Sec: regular stabilizers}
With our set up as before, we have a Cartesian diagram
\[
\begin{tikzcd}
 \widetilde{\fg}^{reg}_1\ar[r, "\widetilde{\chi}_1"]\ar[d,"\pi"]&\fa\ar[d]\\
\fg^{reg}_1\ar[r,"\chi_1"]&\fa/W_\fa.
\end{tikzcd}
\]
The space $\fa/W_\fa$ is the moduli space of regular $G_0$-orbits. We shall introduce a new space which parameterizes \emph{regular stabilizers}. 

In their study of the moduli of $G$-Higgs bundles \cite{DonGaits}, Donagi and Gaitsgory introduce the moduli space of regular centralizers $\overline{G/N}$, where $N$ is the normalizer of a fixed maximal torus $T$. This is a partial compactification of the space of Cartan subalgebras of $\fg$ and is a smooth subscheme of the Grassmannian $\mathrm{Gr}^r(\fg)$of $r$-planes in $\fg$. It comes equipped with a natural smooth morphism 
\[
\varphi: \fg^{reg}\to \overline{G/N} 
\]
which sends $X\in \fg^{reg}$ to its centralizer.
 There is a ramified $W=T\backslash N$-cover $\overline{G/T}\to \overline{G/N}$, where
\[
\overline{G/T}:=\{(\fc,\fb) \in \overline{G/N}\times \Fl_G: \fc\subset \fb\}.
\]
We refer to \cite[Section 2]{DonGaits} for the  definition of a $W$-cover. We remind the reader that a Borel subalgebra is defined to be a Lie algebra of a Borel subgroup.  This is a partial compactification of the quotient map $G/T\to G/N$, which corresponds to restricting to the regular semi-simple locus. By the proof of \cite[Prop. 1.5]{DonGaits}, there exists a Cartesian square
\[
\begin{tikzcd}
 \widetilde{\fg}^{reg}\ar[r]\ar[d]&\overline{G/T}\ar[d]\\
\fg^{reg}\ar[r,"\varphi"]&\overline{G/N}.
\end{tikzcd}
\]
This has the consequence that the $W$-cover $\overline{G/T}\to \overline{G/N}$ is \'{e}tale-locally isomorphic to the $W$-cover $\ft\to \ft/W$. In the next section, we prove a relative version of Theorem 11.6 in \cite{DonGaits}, which gives an isomorphism between two commutative group schemes over $\overline{G/N}$. This isomorphism was used in a fundamental way in \cite{Ngo06}, who worked over the base $\ft/W$ rather than $\overline{G/N}$. The \'{e}tale-local isomorphism \cite[Proposition 1.5]{DonGaits} between these two $W$-covers allows for passage between these two bases. The goal of this section is to prove an analogue of this statement in the case of a quasi-split symmetric pair $(\fg,\fg_0)$.

To this end, we assume that the torus $T=Z_G(A)$ is the centralizer of a maximal $\theta$-split torus $A$. Using the pairing $[\cdot,\cdot]:\fg_1\times\fg_1\to \fg_0$, we let $Ab^{r_1}(\fg_1)\subset \mathrm{Gr}_{r_1}(\fg_1)$ denote the closed subscheme of the Grassmannian of $r_1$-planes in $\fg_1$ on which the restriction of $[\cdot,\cdot]$ vanishes identically. Consider the map 
\begin{align*}
\varphi_1:\fg^{reg}_1&\lra Ab^{r_1}(\fg_1)\\ X&\longmapsto \fz_{\fg_1}(X).
\end{align*}
Essentially the same argument of \cite[Section 10.1]{DonGaits} applies to show that this is a well defined morphism of schemes.
\begin{Prop}\label{Prop: smoothness}
The map $\varphi_1:\fg_1^{reg}\to  Ab^{r_1}(\fg_1)$ is smooth.
\end{Prop}
\begin{proof}

Set $\fc=\varphi_1(x)\in Ab^{r_1}(\fg_1)$. Using the definition of $Ab^{r_1}(\fg_1)$, we may express the tangent space $T_\fc(Ab^{r_1}(\fg_1))$ as the space of maps $T:\fc\to \fg_1/\fc$ such that
\begin{equation}\label{eqn: lie algebra}
[T(y_1),y_2]+[y_1,T(y_2)]=0
\end{equation}
for all $y_1,y_2\in \fc$. To see this, we have by definition that
\[
T_\fc(Ab^{r_1}(\fg_1))= \{\fc'\in Ab^{r_1}(\fg_1\la\ep\ra): p(\fc')=\fa\},
\]
where $\ep^2=0$ and where $p:\fg_1\la\ep\ra\to \fg_1$ is the projection onto the first factor. Any linear map $T:\fc\to \fg_1$ satisfying (\ref{eqn: lie algebra}) gives rise to such an algebra by setting for any $k$-algebra $R$
\[
\fc'_T(R)=\sspan_{R\la\ep\ra}\{ a+\ep T(a): a\in \fc(R)\}.
\]
It is easy to see that $\fc'_{T_1}=\fc'_{T_2}$ if and only if $T_1(y)-T_2(y)\in \fc$ for all $y\in \fc$ and that any $\fc'$ arises in this way. This gives the claimed description. 

In terms of this description, the differential $d\varphi_1: T_x(\fg_1^{reg})\cong \fg_1\to T_{\fc}(Ab^{r_1}(\fg_1))\cong \fg_1/\fc$ sends $v\in \fg_1$ to the unique map $T:\fc\to \fg_1/\fc$ such that
\[
[T(x),y]+[y,v]=0 \quad\text{for all }y\in \fc.
\]
This identity implies that $[T(x)-v,y]=0$ for all $y$ so that we may identify $T(x)\equiv v\pmod{\fc}$. Therefore, letting $\textbf{ev}:T_\fc(Ab^{r_1}(\fg_1))\to \fg_1/\fc$ be the map $T\mapsto T(x)$, we see that the composition
\[
\fg_1\cong T_x(\fg_1^{reg})\xrightarrow{d\varphi_1}T_\fc(Ab^{r_1}(\fg_1))\xrightarrow{\textbf{ev}} \fg_1/\fc
\]
coincides with the tautological quotient map. Finally, the identity $[T(x),y]=-[x,T(y)]$ for all $y\in \fc$ implies that $\textbf{ev}$ is injective, hence an isomorphism. In particular, the image of $\varphi_1$ lies in the smooth locus of $Ab^{r_1}(\fg_1)$ and $d\varphi_1$ is surjective. This proves that $\varphi_1$ is smooth.

\end{proof}

 We define the image of this map to be $\overline{G_0/N_0}$, where $N_0=N_{G_0}(A)\subset G_0$ is the normalizer of $A$ in $G_0$. The following lemma tells us that notation $\overline{G_0/N_0}$ is reasonable.

\begin{Lem}
The $k$-points of $\overline{G_0/N_0}$ parameterizes maximal abelian subalgebras of $\fg_1(k)$ that meet $\fg_1^{reg}(k)$. Moreover, the quotient $G_0/N_0$ embeds as an open subvariety parameterizing Cartan subspaces of $\fg_1$.
\end{Lem}
\begin{Rem}
Since we allow positive characteristic, we remind the reader that a subspace $\fc\subset \fg_1$ is called a \emph{Cartan subspace} if it is a nilpotent subalgebra and such that if $\fg=\fg^0(\fc)\oplus \fg^1(\fc)$ is the decomposition of $\fg$ such that $\fc$ is nilpotent on $\fg^0(\fc)$ and non-singular on $\fg^1(\fc)$, then 
\[
\fg^0(\fc)\cap \fg_1=\fc.
\] In good characteristics, it is a theorem of Levy \cite[Theorem 2.11]{Levy} that such subspaces are maximal toral subspaces of $\fg_1$.
\end{Rem}

\begin{proof}
Let $X\in \fg_1^{reg}$ have centralizer $\fc=\fz_{\fg}(X)$, which is a maximal abelian subalgebra of $\fg$. As this is $\theta$-stable, it decomposes $\fc=\fc_0\oplus\fc_1$, where $\fc_0\cong\Lie(Z_{G_0}(X))$ \cite[Lemma 4.2]{Levy}. Then $\fc\mapsto \fc_1$ gives $\varphi_1(X)$. The maximality follows from the regularity of $X$. Moreover, if we are given such an abelian subalgebra $\fc'\subset \fg_1$, then it is contained in the centralizer of any regular element $X\in \fc'$. Therefore, $\fc'\subset \fz_{\fg}(X)_1$ and maximality forces equality.

 It is known that the quotient $G_0/N_0$ parameterizes Cartan subspaces \cite[Theorem 2.11]{Levy}, and the embedding is obvious.
\end{proof}
We remark that the proof of Proposition \ref{Prop: smoothness} did not rely on the symmetric space being quasi-split. Taking this into account gives a commutative diagram
\[
\begin{tikzcd}
{\fg}_1^{reg}\ar[r]\ar[d,"\varphi_1"]&{\fg}^{reg}\ar[d,"\varphi"]\\
\overline{G_0/N_0}\ar[r]&\overline{G/N},
\end{tikzcd}
\]
where the bottom arrow is given by $\fc\mapsto \fz_{\fg}(\fc).$
We note that the vertical arrows are smooth. We now define $\overline{G_0/T_0}\subset \overline{G_0/N_0}\times \Fl_G$ to be the space of pairs 
\[
(\fa,\fb), \:\fa\subset \fb,\;
\]
 under the restriction that $\fb$ is maximally split, which we recall means that $\fb(\theta)=\fb\cap\theta(\fb)$ is a regular $\theta$-stable Borel subalgebra of $\fz_{\fg}(\fa_{ss})$. Here $\fa=\fa_{ss}\oplus \fa_{nil}$ is the Jordan decomposition of the algebra $\fa$. 
As before this comes equipped with a natural closed immersion $\overline{G_0/T_0}\subset \overline{G/T}$. This may be constructed as follows: we have the diagram
\[
\begin{tikzcd}
\fgreg\ar[r,"\phi"]\ar[d,"\pi"]& \overline{G_0/N_0}\times_{\overline{G/N}}\overline{G/T}\ar[d]\ar[r]&\overline{G/T}\ar[d]\\
\fg_1^{reg}\ar[r,"\varphi"]&\overline{G_0/N_0}\ar[r]&\overline{G/N}, 
\end{tikzcd}
\]
where the arrow $\phi:\fgreg\to  \overline{G_0/N_0}\times_{\overline{G/N}}\overline{G/T}$ is given by
\[
\phi(X,B)=\left(\fz_{\fg_1}(X),(\fz_{\fg}(X),\fb)\right).
\]
Then $ \overline{G_0/T_0}$ is given by the image of the top row of arrows, and we have the following theorem.
\begin{Thm}\label{Thm: local etale}
The diagram
\[
\begin{tikzcd}
 \widetilde{\fg}_1^{reg}\ar[r,"\overline{\varphi}_1"]\ar[d]&\overline{G_0/T_0}\ar[d]\\
\fg_1^{reg}\ar[r,"\varphi"]&\overline{G_0/N_0}.
\end{tikzcd}
\]
is Cartesian. In particular, the $W_\fa$-covers $\fa\to\fa/W_\fa$ and $\overline{G_0/T_0}\to \overline{G_0/N_0}$ are \'{e}tale-locally isomorphic.
\end{Thm}
\begin{proof}
Note that we have a morphism $\fgreg\lra \fg_1^{reg}\times_{\overline{G_0/N_0}}\overline{G_0/T_0}$ given by
\[
(X,B)\longmapsto \left(X,(\fz_{\fg_1}(X),\fb)\right).
\] 
There is clearly a map the other direction, namely the map which sends a triple $\left(X,(\fz_{\fg_1}(X),\fb)\right)$ to $(X,B)$, where $B$ is the unique Borel subgroup with Lie algebra $\fb$. This is obviously an inverse map on geometric points, which suffices to show it is an isomorphism since $\fgreg$ is smooth, hence reduced.
 
Now, we show that the diagram of Cartesian squares
\[
\begin{tikzcd}
\fa\ar[r,leftarrow,"\widetilde{\chi}_1"]\ar[d]& \widetilde{\fg}_1^{reg}\ar[r,"\overline{\varphi}_1"]\ar[d]&\overline{G_0/T_0}\ar[d]\\
\fa/W_\fa\ar[r,leftarrow,"\chi_1"]&\fg_1^{reg}\ar[r,"\varphi_1"]&\overline{G_0/N_0},
\end{tikzcd}
\]
implies that $\overline{G_0/T_0}\to\overline{G_0/N_0}$ is \'{e}tale-locally (with respect to \'{e}tale covers of $\overline{G_0/N_0}$) a pullback $\fa\to \fa/W_\fa$. A similar argument proves that $\fa\to \fa/W_\fa$ is \'{e}tale-locally a pull-back of $\overline{G_0/T_0}\to\overline{G_0/N_0}$. The smoothness of the horizontal arrows implies that for any $x\in \fg^{reg}_1$, we may find a suitable affine open neighborhood $x\in U$  and an affine neighborhood $V\subset \overline{G_0/N_0}$ containing ${\varphi}_1(x)$ such that there is a commutative diagram
\[
\begin{tikzcd}
 {\fg}_1^{reg}\ar[d]&U\ar[l]\ar[r,"\pi"]\ar[d]&\mathbb{A}_V^k\ar[ld,"p"]\\
\overline{G_0/N_0}&{V}\ar[l],
\end{tikzcd}
\]
 for some integer $k$. Here ${\varphi}_1|_U=p\circ\pi$ and $\pi$ is \'{e}tale \cite[Lemma 28.34.20]{stacks-project}. Using the zero section splitting $V\to \mathbb{A}^r_V$, for any $x\in \overline{G_0/N_0}$, we obtain an \'{e}tale neighborhood $V'=U\times_{\mathbb{A}^k_V}V\to \overline{G_0/N_0}$ of $x$ equipped with a locally-closed immersion $V'\to U\to \fg_1^{reg}$ such that the diagram 
\begin{equation}\label{eqn: commutes}
\begin{tikzcd}
 {\fg}_1^{reg}\ar[d]&V'\ar[l]\ar[ld]\\
\overline{G_0/N_0}
\end{tikzcd}
\end{equation}
commutes. Forming the fiber products $V'\times_{\fg_1^{reg}}\fgreg$ and $V'\times_{\overline{G_0/N_0}}\overline{G_0/T_0}$, the commutativity of (\ref{eqn: commutes}) implied that the natural map
\[
V'\times_{\fg_1^{reg}}\fgreg\to V'\times_{\overline{G_0/N_0}}\overline{G_0/T_0}
\]
is an isomorphism. 
Labeling $U'=V'\times_{\overline{G_0/N_0}}\overline{G_0/T_0}$, the $W_\fa$-cover $U'\to V'$ is thus a pullback of $\fa\to \fa/W_\fa$ by Theorem \ref{Thm: main theorem}.
\end{proof}
\begin{Ex}
For the case $(\fg,\fg_0)=(\mathfrak{sl}(2),\mathfrak{so}(2))$, it is shown in \cite[1.6]{DonGaits} that $\overline{G/N}\cong\mathbb{P}^2$, $\overline{G/T}\cong \mathbb{P}^1\times\mathbb{P}^1$ with the map
\begin{align*}
\mathbb{P}^1\times\mathbb{P}^1&\to \mathbb{P}^2\\
              ([x_1:x_2],[y_1:y_2])&\mapsto [x_1y_2+x_2y_1:x_1y_1:x_2y_2].
\end{align*}
The involution induced on $\mathbb{P}^2$ is $[a:b:c]\mapsto [-a:b:c]$. It is easy to see that $\overline{G_0/T_0}\cong\overline{G_0/N_0}\cong \mathbb{P}^1$ with $\mathbb{P}^1\to \mathbb{P}^1$ being the unique degree two map ramified over $0$ and $\infty$. These points correspond to the two nilpotent centralizers contained in $\fg_1$.
\end{Ex}

\subsection{Sheaves of abelian groups}\label{Section: abelian groups}

The final goal of this section is to prove a relative analogue of Theorem 11.6  in \cite{DonGaits}. This is an isomorphism between the tautological sheaf of regular stabilizers on $\overline{G_0/N_0}$ and a certain subsheaf of the restriction of scalars from $\overline{G_0/T_0}$, and will be useful in any attempt to generalize the results of Ng\^{o} \cite{Ngo06} to the case of a relative trace formula associated to a symmetric variety.

The first sheaf to consider is the sheaf of $\theta$-fixed stabilizers $\mathcal{C}_0\subset G_0\times \overline{G_0/N_0}$ given by
\[
\mathcal{C}_0=\{(g,\fa): Ad(g)x=x \text{ for all }x\in \fa\}.
\]
For the second group scheme, let $T$ denote the universal Cartan of $G$. 
As noted in Corolllary \ref{Cor: universal torus}, the torus $T$ may be equipped with a canonical involution $\theta_{can}:T\to T$. Let
\[
T_0:=T^{\theta_{can}}
\]
be the fixed points of this involution. Note that the neutral component $T_0^\circ$ is a torus, but we wish to consider the entire fixed-point subgroup. For example, if $(\fg,\fg_0)$ is split, then this is a finite subgroup. This component group will play a role in the study of relative trace formulae associated to split involutions. 

We also consider the group scheme $\calT_0$ over $\overline{G_0/N_0}$ defined as
\[
{\calT_0}=\left(\mathrm{Res}_{\overline{G_0/T_0}/\overline{G_0/N_0}}(T_0)\right)^{W_{\fa}}.
\]
That is, for any $\overline{G_0/N_0}$-scheme $S$
\[
\calT_0(S)=\Hom_{W_\fa}(\widetilde{S}_0,T_0),
\] 
where $\widetilde{S}_0=S\times_{\overline{G_0/N_0}}\overline{G_0/T_0}$. This functor is representable by a group scheme, giving our $\calT_0$.
\begin{Lem}\cite[Lemmas 2.1,2.2]{KnopAutomorphisms}
The group scheme $\calT_0$ exists and is a smooth, commutative affine group scheme over $\overline{G_0/N_0}$.
\end{Lem}
We have the following analogue of \cite[Theorem 11.6]{DonGaits}.
\begin{Thm}\label{Thm: isomorphism}
There is an isomorphism of smooth commutative group schemes $\iota: \mathcal{C}_0\iso \calT_0$
\end{Thm}
 We are currently working under the assumption that $G_{der}$ is simply connected. In Section \ref{Section: not sc}, we explain how to extend the result to the general case.
\begin{proof}
Recall the isomorphism $\iota: \mathcal{C}\iso \calT$ over $\overline{G/N}$ \cite{DonGaits}, where 
\[
\mathcal{C} = \{(g,\fc)\in G\times \overline{G/N}: Ad(g)x=x \text{ for all }x\in \fc\}
\]
is the canonical centralizer group scheme and
\[
\calT= \left(\mathrm{Res}_{\overline{G/T}/\overline{G/N}}(T)\right)^{W}.
\] This morphism is defined as follows: for any $\overline{G/N}$-scheme $S$, we take an $S$-point of $\mathcal{C}$ to the composition
\[
\widetilde{S}=S\times_{\overline{G/N}}\overline{G/T}\to \mathcal{C}\times_{\overline{G/N}}\overline{G/T}\xrightarrow{\iota'} T,
\]
which is an arrow $S\to \calT$. 
 On geometric points, the isomorphism with $\calT$ takes $(g,\fa)\in \mathcal{C}$ to the $W$-equivariant map
\begin{align*}
\iota(g,\fa): \Fl_G^\fa&\to T\\
			\fb&\to g\pmod{[B,B]},  
\end{align*}
where $\Fl_G^\fa$ is the fiber over $\fa$ in $\overline{G/T}$ the reduced subscheme of which consists of the relevant Borel subalgebras, and $\Lie(B)=\fb$. 

We are interested in the fiber products
\[
\begin{tikzcd}
\mathcal{C}':=\overline{G_0/N_0}\times_{\overline{G/N}}\mathcal{C}\ar[r]\ar[d]&\mathcal{C}\ar[d]\\
\overline{G_0/N_0}\ar[r]&\overline{G/N},
\end{tikzcd}
\]
and the corresponding diagram defining $\calT':=\calT\times_{\overline{G/N}}\overline{G_0/N_0}$. Then we have $\iota:\mathcal{C}'\iso \calT'$ is an isomorphism of smooth group schemes over $\overline{G_0/N_0}$. There is a natural involution on $\mathcal{C}$ by restricting the involution $\theta(g,\fa)= (\theta(g),\theta(\fa))$ on $G\times\overline{G/N}$ to $\mathcal{C}$. This naturally induces an involution on $\mathcal{C}'$ given by
\[
\theta(g,\fa)=(\theta(g),\fa).
\]
In particular, the fixed-point subgroup scheme is precisely $\mathcal{C}_0$. By \cite[Proposition 3.4]{EdixhovenNeron}, it follows that $\mathcal{C}_0$ is smooth over $\overline{G_0/N_0}$.
 The corresponding involution on $\calT'$ sends $\iota(g,\fa)$ to $\iota(\theta(g),\fa)$ so that $\iota$ induces an isomorphism
 \[
 \iota:\mathcal{C}_0\iso (\calT')^\theta.
 \]
\begin{Lem}\label{Lem: fixed subgroup}
With respect to this involution, there is an isomorphism $(\calT')^\theta\iso \calT_0$.
\end{Lem}

\begin{proof}
We first construct the map. Let $S$ be a $\overline{G_0/N_0}$-scheme and let $x:S\to \mathcal{C}_0$ be a $\theta$-fixed point. The corresponding $S$-point of $\calT$ is a $W$-equivariant map
\[
\varphi_x:\widetilde{S}=S\times_{\overline{G/N}}\overline{G/T}\to T.
\]
Note that there is a natural inclusion
\[
\widetilde{S}_0=S\times_{\overline{G_0/N_0}}\overline{G_0/T_0}\hra S\times_{\overline{G_0/N_0}}\left(\overline{G_0/N_0}\times_{\overline{G/N}}\overline{G/T}\right)=\widetilde{S},
\]
so that by restriction we have a morphism $\varphi_x:\widetilde{S}_0\to T$ which is $W_\fa$-equivariant. It remains to show that the image lies in $T_0\subset T$. For each geometric point $s\in S$ let $x(s)=(g,\fa)\in \mathcal{C}_0$ be the corresponding geometric point of $\mathcal{C}_0$. The gives rise to a map $(\Fl_G^\fa)_{split}\to T$ given by
\[
{\varphi_x}(\fb)=t_\fb =g\pmod{[B,B]},
\]
 for all maximally split Borel subgroups with $\fa\subset \fb=\Lie(B)\in(\Fl_G^\fa)_{split}$. Since $B$ is maximally split,  the proof of Proposition \ref{Prop: canonical involution} implies we may choose $h\in Z_{G}(\fa_{ss})$ such that $B$ is $\theta^h$-split. Since $g\in Z_{G}(\fa)$, if we write $g=tn$ for the Jordan decomposition, then $t\in Z(Z_{G}(\fa_{ss}))$. This follows from the corresponding fact about centralizers of regular nilpotent elements and \cite[Theorem 7]{KR71}. We may now compute
\begin{align*}
\theta_{can}(t_\fb)&=\theta^h(g)\pmod{[B,B]}\\
			&=\theta^h(t)\pmod{[B,B]}\\
			&=\theta(t)\pmod{[B,B]}\\&=g\pmod{[B,B]}=t_\fb,
\end{align*}
where we used the fact that $x(s)=(g,\fa)\in \mathcal{C}_0$ is a fixed point. Therefore, the morphism $\varphi_x:\widetilde{S}_0\to T$ factors through the inclusion of $T_0\subset T$, and we have a morphism $(\calT')^\theta\lra \calT_0$. 

We now show that this morphism is an isomorphism. Using \cite[Proposition 3.1]{EdixhovenNeron}, which in particular implies that $\overline{G_0/N_0}\subset \overline{G/N}$ is a closed immersion, we see that
\[
(\mathcal{T}')^\theta\subset \mathcal{T}'\subset \mathcal{T}
\]
are all closed immersions. In particular, the morphism $(\calT')^\theta\lra\calT_0$ is closed. Since $\calT_0$ is smooth (hence reduced), it suffices to check that this is an isomorphism over the regular semi-simple locus. Note that $W$-equivariance implies that for any $S\to \overline{G_0/N_0}$, a morphism $\widetilde{S}_0\to T_0$ determines a unique morphism $\widetilde{S}\to T$. This is because
\[
W\times^{W_\fa}{G_0/T_0}\iso {G_0/N_0}\times_{{G/N}}{G/T}, 
\]
where the map is given on geometric points by $[(w,gT_0)]\mapsto (gN_0,gw^{-1}T)$. This gives a natural map $\calT_0\to \calT'$. 

Since the map $Z_{G}(\fa)\to B/[B,B]$ is injective over the regular semi-simple locus, the previous argument implies that $\theta(g)=g$. This implies that the above morphism factors through $\calT_0\to (\calT')^\theta$, and it gives an inverse morphism on this locus. This shows that $(\calT')^\theta\to \calT_0$ is an isomorphism.
\end{proof}
This completes the proof of Theorem \ref{Thm: isomorphism}. Indeed we already have seen that $\mathcal{C}_0$ is smooth and that there is an isomorphism $\mathcal{C}_0\iso (\calT')^\theta$.
\end{proof}
Given the inclusion of subgroups $T_1:=T^{-\theta_{can}}\subset T$, we may form the following subgroup scheme of $\mathcal{C}$ over $\overline{G_0/N_0}$: 
\[
\mathcal{C}_1=\{(g,\fa)\in \mathcal{C}: \theta(g)=g^{-1}\},
\] 
We may similarly define $T_1\subset T$ and form the corresponding $W_\fa$-invariant restriction of scalars group schemes $\calT_1$.

\begin{Cor}
We also have isomorphisms $\mathcal{C}_1\iso \calT_1$.
\end{Cor}
\begin{proof}
The argument above goes through verbatim in this case. We leave the details to the reader.  

\end{proof}

\subsection{When $G_{der}$ is not simply connected}\label{Section: not sc}
In \cite{DonGaits}, the authors do not assume that $G_{der}$ is simply connected. That they work in full generality is of the utmost importance for applications to the Langlands program. In this subsection, we describe the analogous result in the symmetric space setting when we relax the simple-connectedness assumption. 

 Donagi and Gaitsgory first define $$\overline{\calT}=\left(\mathrm{Res}_{\overline{G/T}/\overline{G/N}}(T)\right)^W,$$ as in the preceding section, then define a subgroup group scheme $\calT\subset \overline{\calT}$ by imposing that certain eigenvalues occur on the branching locus of the map $\overline{G/T}\to \overline{G/N}$ to obtain an isomorphism $\mathcal{C}\iso\calT$. More precisely, let $\Phi=\Phi(\fg,\ft)$ denote the set of roots of $(G,T)$. For any root $\al$ of $T$, let $D_\al\subset \overline{G/T}$ denote the fixed-point locus of the involution $s_\al$. For any $S\to \overline{G/N}$ and $S$-point $t:\widetilde{S}=S\times_{\overline{G/N}}\overline{G/T}\to T$ of $\overline{\calT}$, the composition 
\[
\qquad\qquad\qquad\qquad\qquad\qquad\qquad S\times_{\overline{G/N}}D_\al\hra \widetilde{S}\xrightarrow{t} T\xrightarrow{\al}\Gm\qquad\qquad\qquad\qquad\qquad\qquad(C_\al)
\] 
has image $\pm1$. The group subscheme $\calT$ is defined to be the subgroup of maps avoiding $-1$, which as a short-hand we call condition $(C_\al)$. They then show that $\mathcal{C}\iso \calT$.

Under the assumption that $G_{der}$ is simply connected, this subscheme is actually the entire group $\overline{\calT}$. Nevertheless, the argument in the proof of Theorem \ref{Thm: isomorphism} did not depend on this restriction, so to generalize we need only explicate the appropriate restrictions on the points of the group scheme $\calT_0$ for Lemma \ref{Lem: fixed subgroup} to hold. 

To make this precise, we drop the assumption that $G_{der}$ is simply connected and now set $$\overline{\calT}_0=\left(\mathrm{Res}_{\overline{G_0/T_0}/\overline{G_0/N_0}}(T_0)\right)^{W_{\fa}},$$
and describe a subgroup scheme $\calT_0\subset \overline{\calT}_0$ such that we have an isomorphism $\mathcal{C}_0\iso \calT_0$.
For each $\al\in \Phi$, we form the fiber product $D_\al^\theta=\overline{G_0/T_0}\times_{\overline{G/T}}D_\al$. This is never empty since it contains the pairs $(\fa,\fb)$ where $\fa$ is nilpotent, for example. Then for any scheme $S\to \overline{G_0/N_0}$, the proof of Lemma \ref{Lem: fixed subgroup} makes clear that for an element $t\in \calT^\theta(S)$, we have a commutative diagram
\[
\begin{tikzcd}
S\times_{\overline{G/N}}D_\al\ar[r]&\widetilde{S}\ar[r,"t"]&T\ar[dr,"\al"]&\\
S\times_{\overline{G_0/N_0}}D_\al^\theta\ar[r]\ar[u]&\widetilde{S}_0\ar[u]\ar[r,"t"]&T_0\ar[u]\ar[r,"\al"]&\Gm.
\end{tikzcd}
\]
Since $t$ satisfies the condition $(C_\al)$, we conclude that the composition $\lam\circ t: S\times_{\overline{G_0/N_0}}D_\al^\theta\to \Gm$ avoids $-1$. In particular, we have the following characterization.
\begin{Cor}
Define subgroup $\calT_0\subset \overline{\calT}_0$ so that for any $\overline{G_0/N_0}$-scheme $S$, the set of $S$-points $\calT_0(S)$ consists of $W_\fa$-equivariant arrows $t: \widetilde{S}_0\to T_0$ such that for every $\al\in \Phi$ the composition
\begin{equation*}
S\times_{\overline{G_0/N_0}}D_\al^\theta\hra\widetilde{S}_0\xrightarrow{t}T_0\xrightarrow{\al|_{T_0}}\Gm
\end{equation*}
avoids $-1\in \Gm$. Then we have an isomorphism $\mathcal{C}_0\iso \calT_0$.
\end{Cor}

\begin{Ex}
In the case of $(\fsl(2),\fso(2))$, we need only consider one root $\al:T\to \Gm$. In this case, $$D^\theta_{\al}=\Spec k\sqcup \Spec k=0\sqcup \infty$$ is the disjoint union of points corresponding to the two nilpotent regular centralizers contained in $\fg_1$ and  associated $\theta$-stable Borel subalgebras. 

Working with $G=\SL(2)$ gives $T_0=Z(G)= \{\pm Id\}$. For either nilpotent closed point $n$, $n \times_{\overline{G_0/N_0}}D_\al^\theta = \Spec k$ is the corresponding pair and there are two morphisms $t:\Spec k \to T_0$. Since $\al(\pm I)=1$, both are admissible and we find
$
(\mathcal{C}_0)_n \iso \{\pm 1\}.
$

On the other hand, if we work with $G=\PGL(2)$, then $T_0 = \{\omega(\pm 1)\}$, where $\omega:\Gm \to T$ is the fundamental coweight. While there are two maps $t:\Spec k \to T_0$, only the one with image $Id=\omega(1)$ is admissible since $\al(\omega(-1))=-1$. Thus $(\mathcal{C}_0)_n \iso \{1\}$ in this case.
\end{Ex}

\section{Smoothness and resolution of singularities}\label{Sec: Smoothness}
In this final section, we consider the question of whether $\fgreg$ has a partial compactification $\fgres\subset \fgbul$ that plays a role analogous to the Grothendieck-Springer resolution over the entire space $\fg_1$. That is, we ask if there is a smooth family of resolutions of the singularities of the adjoint quotient map. For simplicity, we assume now that $G$ is semi-simple and continue to assume that it is simply connected (see \cite[9.16]{Steinberg} and \cite[Lemma 1.3]{Levy}).

Toward this question, we consider a subspace which we show recovers the classical Grothendieck-Springer resolution in the case of the diagonal symmetric space $(\fg_0\oplus \fg_0,\Delta \fg_0)$. We also show that our proposal does indeed form a family of resolutions of the singularities of the quotient map $\fg_1\to \fg_1//G_0$, and give a sufficient criterion for this space to be smooth.

However, there are very basic cases when the morphism $\chi_1:\fg_1\to \fg_1//G_0$ does not admit a simultaneous resolution after base change to any finite ramified cover of $\fg_1//G_0$.  In such cases, our space $\fgres$ will not give rise to an irreducible scheme. For example, assume that $k=\cc$ so that we may work topologically.  If we consider the split involution of type $A$ associated to the symmetric pair $(\fsl(n),\fso(n))$ ($n>2$) we may see that no simultaneous resolution exists as follows: consider the subregular Slodowy slice  $S\subset \fg_1$ studied in \cite{ThorneVinberg}. Then $f: S\to \fg_1//G_0$ is a family of plane curves with an isolated singularity at $0$ of type $A_n$. The monodromy representation on $R^1f_\ast\zz$ has image the principle congruence subgroup $\Gamma(2)\subset \Sp_{2g}(\zz)$, where $g$ is the genus of the curves \cite{arnol1988singularities}, so no finite base change can remove this obstruction. Since a simultaneous resolution of $\fg_1\to \fg_1//G_0$ would pull back to one of $S\to \fg_1//G_0$, it follows that no such resolution can exist. The author wishes to thank Jack Thorne for explaining this example to him.


Our proposal for $\fgres$ is quite natural: we simply extend the construction of $\fgreg$ from Proposition \ref{Prop: regular characterization} to all of $\fg_1$ and consider the following subspace of $\widetilde{\fg}_1$:
\begin{equation*}
\widetilde{\fg}^{res}_1:=\{(X,B)\in \fg_1\times \mathcal{F}l_G : B(\theta) = Z_B(X_{ss}) \text{ is a regular $\theta$-stable Borel of }Z_G(X_{ss}) \},
\end{equation*}
where the superscript $res$ stands for resolution. The next proposition shows that this construction recovers the Grothendieck-Springer resolution for the diagonal symmetric space.

\begin{Prop}\label{Prop: recovers groth}
Consider the diagonal symmetric space $(\fg_0\times \fg_0,\Delta \fg_0)$. Then 
\begin{align*}
\phi:\widetilde{\fg}^{res}_1&\lra \widetilde{\fg}_0\\
	((X,-X),(B_1,B_2))&\mapsto (X,B_1)
\end{align*}
is an isomorphism, where this latter variety is the Grothendieck-Springer resolution of $\fg_0$.
\end{Prop}
\begin{proof}
First, note that the property of $((X,-X),(B_1,B_2))$ lying in $\widetilde{\fg}^{res}_1$ is that 
\[
B_1\cap B_2= Z_{B_1}(X_{ss})=Z_{B_2}(X_{ss}),
\]
since $(X,-X)_{ss}=(X_{ss},-X_{ss})$ so that 
\[
Z_{G\times G}((X,-X)_{ss})=Z_{G}(X_{ss})\times Z_{G}(X_{ss}).
\]

We construct an inverse to $\phi$: Let $(X,B)\in \widetilde{\fg}_0$ and suppose that $X=X_{ss}+X_{nil}$. 
Consider the parabolic subgroup $P(X)=Z_G(X_{ss})B\supset B$ with Levi subgroup $Z_G(X_{ss})$. Note that if $P(X)=Z_G(X_{ss})U_P$ is the Levi decomposition of $P(X)$, then $B=Z_B(X_{ss})U_P$. It is standard theory that there exists a unique parabolic subgroup $P(X)^{op}$ such that $P(X)\cap P(X)^{op}=Z_{G}(X_{ss})$; let $U_P^{op}$ be its unipotent radical.
Then, the group $B_X^{op}=Z_{B}(X_{ss})U_P^{op}$ is also a Borel subgroup of $G$.
By construction, $B\cap B_X^{op}=Z_{B}(X_{ss})$. Thus, we define the morphism
\begin{align*}
\psi:\widetilde{\fg}_0&\lra \widetilde{\fg}^{res}_1\\
		(X,B)&\mapsto ((X,-X),(B,B_X^{op})).
\end{align*}
Clearly, $\phi\circ\psi=Id$. We claim also that $\psi\circ\phi=Id$. Suppose that $$\psi\circ\phi((X,-X),(B,B_1))=((X,-X), (B,B_2)).$$ This implies that
\[
B\cap B_1 = B\cap B_2 =Z_B(X_{ss}).
\]
The Borel subgroup $Z_B(X_{ss})$ contains a maximal torus $S$ centralizing $X_{ss}$, so $B_1=wB_2w^{-1}$ for some $w\in W_S$. The claim now follows since, for fixed Borel subgroup $B$ containing a maximal torus $S$, the set of subgroups $B\cap B'$ as $B'$ ranges over the $W_S$-torsor of Borel subgroups containing $S$ are all distinct. This final statement is true as the sets $\Phi^+_w=\{\al\in \Phi^+: w\al<0\}$ for $w\in W_S$ are distinct subsets of $\Phi^+$.
\end{proof}

We now consider the fibers of the map $\widetilde{\chi}_1:\widetilde{\fg}^{res}_1\to \fa$. Let $a\in \fa$, and recall the Kostant-Weierstrass section $\kappa:\fa/W_\fa\to \fg_1$, which depends on a choice of regular nilpotent element. Setting $X(a)=\kappa(\overline{a})_{ss}$, we have the identification
\[
\chi_1^{-1}(\overline{a})_{red} \cong G_0\times^{Z_{G}(a)^\theta}\left(X(a)+\caln(a)_1\right),
\]
where $\caln(a)_1=\caln(\fz_\fg(a))_1$ is the nilpotent cone in the $(-1)$-eigenspace of $\fz_\fg(a)$. This scheme decomposes into finitely many irreducible components $\caln(a)_1=\cup_i\caln(a)_1^i$. Since $G_0$ and $Z_{G_0}(a)$ are connected, we have a decomposition into irreducible components
\[
\chi_1^{-1}(\overline{a})_{red} =\bigcup_{i\in \pi_0(\caln(a)_1)}{\chi}_1^{-1}(\overline{a})_{i}
\]
where
$
{\chi}_1^{-1}(\overline{a})_{i}\cong G_0\times^{Z_G(a)^\theta}\caln(a)_1^i.
$

\begin{Thm}\label{Thm: resolution fibers}
There is a decomposition into connected components
\[
\widetilde{\chi}_1^{-1}(a)_{red}=\bigsqcup_{i\in \pi_0(\caln(a)_1)}\widetilde{\chi}_1^{-1}(a)_{w(i)}
\]
such that each component is smooth and the map $\widetilde{\chi}_1^{-1}(a)_{w(i)}\to {\chi}_1^{-1}(\overline{a})_{i}$ is a resolution of singularities. In particular, $\widetilde{\chi}_1^{-1}(a)_{red}$ is smooth.
\end{Thm}
\begin{proof}
Recall that $\theta|_{Z_G(a)}$ is a quasi-split involution, which we also denote by $\theta$. Let $Z_G(a)^\theta$ denote the fixed point subgroup of $\theta$ in $Z_G(a)$. Note also that $Z_G(a)^\theta = Z_{G_0}(a)=Z_G(a)\cap G_0$ is connected since the derived subgroup $Z_G(a)^{(1)}$ is simply connected \cite{Steinberg}.

By \cite[Proposition 2.3.4]{reeder1995}, the fixed point set of $\Fl_{Z_G(a)}$ is a disjoint union of varieties isomorphic to $\Fl_{Z_{G_0}(a)}$. The above morphism only maps to the \emph{regular} $\theta$-stable Borel subgroups of $Z_G(a)$, denoted by $(\Fl_{Z_G(a)}^\theta)^{reg}$. Using the notation from Proposition \ref{Prop: nilp fixed points}, we have
\[
(\Fl_{Z_G(a)}^\theta)^{reg}=\bigsqcup_{i\in \pi_0(\caln(a)_1)}C_{w(i)},
\]
where $C_{w(i)}=\{B\in\Fl_{Z_G(a)}: \Lie(B)\cap (\caln(a)_1^i)^{reg}\neq \emptyset\}$ is the closed $Z_G(a)^\theta$-orbit of regular $\theta$-stable Borel subgroups whose Lie algebras meet the regular locus of the component $\caln(a)_1^i\subset\caln(a)_1$. 

 For simplicity, we adopt the notation ${}^gB=g^{-1}Bg$. Let $(X,B)\in \widetilde{\chi}_1^{-1}(a)$. Then there exists $g\in G_0$ and $n\in \caln(a)_1$ such that $X=\Ad(g)(X(a)+n)$ so that $(X(a)+n,{}^gB)\in \widetilde{\chi}_1^{-1}(a)$. If $g'\in G_0$ is another element such that $X=\Ad(g')(X(a)+n')$, then $(X(a)+n',{}^{g'}B)\in \widetilde{\chi}_1^{-1}(a)$ and 
\[
g^{-1}g'\in Z_G(a)\cap G_0=Z_G(a)^\theta,\quad\text{and}\quad n=\Ad(g^{-1}g')(n').
\]
Then $$({}^{g}B)(\theta) = (g^{-1}g'){}^{g'}B(\theta)(g^{-1}g')^{-1},$$ so that the regular $\theta$-stable Borel subgroups ${}^{g}B(\theta), {}^{g'}B(\theta)\subset Z_G(a)$ are in the same $Z_G(a)^\theta$-orbit. Since $Z_G(a)^\theta$ is connected, this implies a decomposition 
\[
\widetilde{\chi}_1^{-1}(a)=\bigsqcup_{i\in \pi_0(\caln(a)_1)}\widetilde{\chi}_1^{-1}(a)_{w(i)}
\]
into connected components. It is clear that the restriction of $\pi_1$ to any component gives a morphism $\widetilde{\chi}_1^{-1}(a)_{w(i)}\to {\chi}_1^{-1}(\overline{a})_{w(i)}$. 

We need the following lemma.
\begin{Lem} 
The image of $\widetilde{\chi}_1^{-1}(a)_{w(i)}$ under the projection $\pi_2:\widetilde{\fg}_1\to \Fl_G$ lies in a single $G_0$-orbit.
\end{Lem}
\begin{proof}
If $(X,B_1),(Y,B_2)\in \widetilde{\chi}_1^{-1}(a)_{w(i)}$, then there exists $g_1,g_2\in G_0$ such that
\[
X=\Ad(g_1)(X(a)+n_1), \quad \text{and}\quad Y=\Ad(g_1)(X(a)+n_2),
\]
and ${}^{g_1}B_1(\theta),{}^{g_2}B_2(\theta)\in C_{w(i)}$. We may assume $X=X(a)+n_1$ so that $g_1=1$. Then since $C_{w(i)}$ is a single $Z_{G}(a)^\theta$-orbit, we find that there is $g_3\in G_0$ such that, replacing ${}^{g_2}B_2$ by ${}^{g_3}B_2$, $B_1(\theta)={}^{g_3}B_2(\theta)$.  Since $(X(a),B_1), (X(a),{}^{g_3}B_2)\in \widetilde{\fg}_1$, Corollary \ref{Cor:regular Borels orbit} thus implies that $B_1$ lies in the same $G_0$-orbit as ${}^{g_3}B_2$ in $\Fl_G$, so that $B_1$ and $B_2$ do as well. 
\end{proof}

Now fix a Borel $B$ such that $(X(a),B)\in\widetilde{\chi}_1^{-1}(a)$ with $B(\theta)\in C_{w(i)}$. Then for every $(X,P)\in \widetilde{\chi}_1^{-1}(a)_{w(i)}$, the previous lemma says that we may write $P={}^{g}B=g^{-1}Bg$ for some $g\in G_0$. This implies that $(\Ad(g)(X),B)\in \widetilde{\chi}_1^{-1}(a)$ so that
\[
\Ad(g)(X)\equiv X(a)\pmod{[\fb(\theta),\fb(\theta)]}.
\]
Thus, the difference $\Ad(g)(X)-X(a)\in[\fb(\theta),\fb(\theta)]$ is nilpotent, implying $\Ad(g)(X)\in X(a)+\fn(\theta)_{1}$. We have proven the following lemma.
\begin{Lem} \label{Lem: fiber iso}
There is an isomorphism
\begin{align*}\label{eqn: component of fiber}
\widetilde{\chi}_1^{-1}(a)_{w(i)}\cong\{(X,gB)\in \fg_1\times G_0\cdot B: X\in \Ad(g)\left(X(a)+\fn(\theta)_{1} \right)\},
\end{align*}
which we may identify with $G_0\times^{B(\theta)_0}\left(X(a)+\fn(\theta)_{1}\right)$. \qed
\end{Lem}

Let us now consider the resolution of singularities of $\caln(a)_1$. Using Proposition \ref{Prop: nilp fixed points}, we see that 
\[
\widetilde{\caln}(a)_1=\{(X, B)\in \caln(a)_1\times \Fl_{Z_G(a)}: X\in \Lie(B),\: B\text{ regular $\theta$-stable Borel}\}
\]
has a similar decomposition into components
\[
\widetilde{\caln}(a)_1=\bigsqcup_{i\in \pi_0(\caln(a)_1)}E_{w(i)}\lra \bigsqcup_{i\in \pi_0(\caln(a)_1)}C_{w(i)}.
\]
Fix a component $\caln(a)_1^i$, and restrict the previous map to the fiber over this component.
By \cite[Proposition 3.2]{reeder1995}, $\pi_i:E_{w(i)}\lra \caln(a)_1^i$ is a resolution of singularities. More explicitly, let $e\in \caln(a)_1^{i,reg}$. In Section \ref{Section: resolutions of nilp}, we constructed a Borel subgroup $P\subset Z_G(a)$ with Lie algebra $\fp=\Lie(P)$ such that if $\fq_i=\caln(a)^i_1\cap [\fp,\fp]$, then $e\in \fq_i$, and 
\[
E_{w(i)}\cong Z_{G}(a)^\theta\times^{P^\theta}(X(a)+\fq_i).
\] 
It follows that
\[
G_0\times^{Z_G(a)^\theta}\pi_i:G_0\times^{Z_G(a)^\theta}\left(Z_{G}(a)^\theta\times^{P^\theta}(X(a)+\fq_i)\right)\lra G_0\times^{Z_G(a)^\theta}(X(a)+\caln(a)_1^i)
\]
is a resolution of singularities of an irreducible component of $\chi_1^{-1}(\overline{a})_{red}$. The natural map 
\[
f_i: G_0\times^{Z_G(a)^\theta}\left(Z_{G}(a)^\theta\times^{P^\theta}(X(a)+\fq_i)\right)\to G_0\times^{P^\theta}(X(a)+\fq_i),
\]
 is an isomorphism. 
For any Borel subgroup $B\subset G$ such that $X(a)+e\in \Lie(B)$ and $B(\theta)=B\cap \theta(B) = P$, we may identify $ \fn(\theta)_{1}=\fq_i$ and $B(\theta)_0=P^\theta$ so that Lemma \ref{Lem: fiber iso} implies that $f_i$ induces an isomorphism
\[
f_i: G_0\times^{Z_G(a)^\theta}\left(Z_{G}(a)^\theta\times^{P^\theta}(X(a)+\fq_i)\right)\xrightarrow{\sim} \widetilde{\chi}_1^{-1}(a)_{w(i)},
\]
and thus a commutative diagram
\[
\begin{tikzcd}
 \widetilde{\chi}_1^{-1}(a)_{w(i)}\ar[r,"\sim"]\ar[d,"\pi_1"]&G_0\times^{Z_G(a)^\theta}\left(Z_{G}(a)^\theta\times^{P^\theta}(X(a)+\fq_i)\right)\ar[d,"G_0\times^{Z_G(a)^\theta}\pi_i"]\\
{\chi}_1^{-1}(\overline{a})_{w(i)}\ar[r,"\sim"]&G_0\times^{Z_G(a)^\theta}\caln(a)_1^i,
\end{tikzcd}
\]
showing that $\pi_1: \widetilde{\chi}_1^{-1}(a)_{w(i)}\to{\chi}_1^{-1}(a)_{w(i)}$ is a resolution of singularities.

\end{proof}

Consider the morphism $\widetilde{\chi}_1:\widetilde{\fg}^{res}_1\to \fa$. We wish to know if $\fgres$ may be endowed with a natural scheme structure such that this morphism is smooth. Our analysis of the fibers of this morphism shows that their reduced subschemes are all smooth of dimension $r_1=\dim(\fa)$. To use our analysis of the fibers to conclude smoothness, we require the following technical lemma. 
\begin{Lem}\label{Lem: technical smooth}
Suppose that $X$ is a variety (that is, a reduced, irreducible, separated scheme of finite type over an algebraically closed field $k$) and suppose $Y$ is a smooth affine $k$-scheme of dimension $m$. Suppose that $f:X\to Y$ is a morphism such that 
\begin{enumerate}

\item\label{(b)} $(X_y)_{red}$ is smooth of fixed dimension $n>0$ for all  $y\in Y(k)$, 
\item\label{(c)} the maximal open $V_y\subset X_y$ which is a reduced scheme is dense in $X_y$ for all $y\in Y(k)$.
\end{enumerate}
Then $f$ is smooth. In particular, $X$ is smooth over $k$.
\end{Lem}
We remark that the statement trivially holds for $n=0$ once one assumes that $f$ is surjective.


\begin{proof}
Denote by $V\subset X$ the open subscheme on which the restriction $f|_{V}$ is smooth. Then $V_y$ is the fiber $(f|_{V})^{-1}(y)$: this follows from \cite[2.8]{deJong}. Let $N:X'\to X$ denote the normalization of $X$; note that $V\subset X'$ is an open subscheme of $X'$ as well. We have the commutative diagram
\begin{equation}\label{eqn: triangle}
\begin{tikzcd}
X'\ar[dr,"f'"]\ar[d,"N"]&\\
X\ar[r,"f"]&Y.
\end{tikzcd}
\end{equation}
First, we show that the assumptions imply that for each $y\in Y(k)$, the induced map $(X'_y)_{red}\xrightarrow{\sim}(X_y)_{red}$ is an isomorphism. Indeed, this is a finite morphism that is an isomorphism over $V_y=(V_y)_{red}$. Moreover, $(X'_y)_{red}$ is equidimensional by Krull's height theorem, so that the map is birational. It is thus an isomorphism as the base is smooth, hence normal. In particular, $f':X'\to Y$ also satisfied the assumptions of the lemma. This also implies a bijection between closed points of $X'$ and $X$.

For any smooth effective Cartier divisor $Z\subset Y$, consider the morphism $(f')^{-1}(Z)\to Z$. Since $X'$ normal, it follows that $(f')^{-1}(Z)$ is reduced \cite[Tags 0344 and 0345]{stacks-project}. If $\dim(Y)=1$, this shows that the fibers of $f'$ are reduced, so that they are smooth by the preceding paragraph. But then $f':X'\to Y$ is a morphism with smooth equidimensional fibers over a smooth base. It is flat by \cite[Theorem 3.3.27]{Schoutens}, and thus smooth by \cite[Theorem 10.2]{Hartshorne}. For $\dim(Y)>1$, the version of Bertini's theorem stated in \cite[Theorem 6.3 (4)]{Jouanolou} implies that for any $y\in Y(k)$ we may choose $Z$ such that $y\in Z(k)$ and $(f')^{-1}(Z)$ is irreducible. Note that we have used the fact that $f'$ is surjective. Then the map $(f')^{-1}(Z)\to Z$ also satisfies (\ref{(b)}) and (\ref{(c)}). By induction on the dimension of the base, $(f')^{-1}(Z)\to Z$ is a smooth morphism. In particular, all the fibers of $f'$ are smooth. By the argument above, $f:X'\to Y$ is a smooth morphism.

To conclude, we show that $N:X'\to X$ is an isomorphism. Since we have seen that it is bijective on closed points, we need only check that it is injective on tangent vectors. The diagram (\ref{eqn: triangle}) implies that any vector in the kernel of $dn$ must be vertical with respect to $f':X'\to Y$; that is, it must lie in $T(X'_y)\subset T(X')$ for some $y\in Y(k)$. But this is impossible since $X'_y=(X'_y)_{red}\xrightarrow{\sim}(X_y)_{red}$ is an isomorphism of smooth varieties.
\end{proof}
\begin{Cor}
If $\widetilde{\fg}^{res}_1$ is a variety, the morphism $\widetilde{\chi}_1:\widetilde{\fg}^{res}_1\to \fa$ is smooth.
\end{Cor}
\begin{proof}
This follows from Lemma \ref{Lem: technical smooth}. To see this, take $f=\widetilde{\chi}_1$, $X=\widetilde{\fg}_1$, and $Y=\fa$. Then under the assumption on $G=\widetilde{\fg}_1$, the spaces $X$ and $Y$ satisfy the criteria, (\ref{(b)}) follows from Theorem \ref{Thm: resolution fibers} above, and (\ref{(c)}) follows from the Cartesian diagram in Proposition \ref{Prop: regular characterization} which implies that 
\[
\left(\widetilde{\chi}^{reg}_1\right)^{-1}(t)\subset \widetilde{\chi}_1^{-1}(t),
\]
which is Zariski open and dense, is smooth. 
\end{proof}

\begin{appendices}

\section{Comparison with Knop's section}\label{Appendix: Knop}
In this appendix, we compare our use of a Kostant-Weierstrass section in Section \ref{Section: relative Groth} to the results and methods of Knop in the context of spherical varieties. We therefore assume that $k$ is an algebraically closed field of characterisic zero, as this is the context of Knop's theory \cite{knop1994asymptotic}.

\subsection{Knop's section} Let us recall the results of \cite[Section 3]{knop1994asymptotic}. For ease of comparison to the Kostant-Weierstrass section, we use a more general set up, following \cite{Sakrank1}. Let $G$ be a connected reductive algebraic group over $k$ and let $X$ be a spherical $G$-variety. For simplicity, we assume that $X$ is quasi-affine.
\begin{Rem}
 Knop's results hold more generally for any normal $G$-variety $X$ that is non-degenerate in the sense of \cite[Section 3]{knop1994asymptotic}.
\end{Rem}

Fix a Borel subgroup $B$ and let $T=B/[B,B]$ denote the canonical torus quotient. Letting $\mathfrak{t}$ denote the Lie algebra of $T$, we have the natural inclusion of dual spaces
\[
\ft^\ast\subset\mathfrak{b}^\ast, \quad \text{where } \fb=\Lie(B),
\]
 as linear maps trivial on the nilpotent radical of $\fb.$ Finally, let $W$ denote the Weyl group of the pair $(G,T)$.
 
 Let $\mathring{X}\subset X$ denote the open $B$-orbit and let $P(X)\supset B$ denote the largest parabolic subgroup of $G$ stabilizing $\mathring{X}.$ It is known that $P(X)$ acts on the categorical quotient $[B,B]\backslash\backslash\mathring{X}$ through a torus quotient $T_X$, known as the \emph{canonical torus} of the variety $X$. This induces a quotient morphism $T\to T_X$ and dual inclusion $\ft_X^\ast\subset \ft^\ast.$ 
 Set $\Fl_G$ for the flag variety of Borel subgroups of $G$, and let $\Fl_X$ denote the flag variety of parabolic subgroups conjugate to $P(X).$ We consider the correspondence variety \[
(X\times \Fl_X)^\circ=\{(x,P)\in X\times \Fl_X : x\text{ lies in the open $P$-orbit}\};
\]this is equipped with natural projections
\[
\begin{tikzcd}
 &(X\times \Fl_X)^\circ\ar[rd,"p_2"]\ar[ld,"p_1",swap]&\\
 X&&\Fl_X.
\end{tikzcd}
\]
\begin{Rem}
 In \cite{knop1994asymptotic}, Knop fixes a Borel subgroup $B$, and works only within the fiber $p_2^{-1}(P(X))=\mathring{X}$. Here $P(X)\supset B$ is the unique parabolic in $\Fl_X$ containing $B.$ As we will see in the next subsection, it is more natural to consider the fibers $p_1^{-1}(x_0)$ when studying the infinitesimal theory of $X$ at a point $x_0\in X$.
\end{Rem}
Consider the cotangent bundle $\pi:T^\ast(X)\to X$. The $G$-action on $X$ gives rise to the \emph{moment map}
\begin{align*}
    \mu: T^\ast X&\lra \fg^\ast\\
    \al\:&\longmapsto l_\al,
\end{align*}
where $l_\al$ is the functional 
\[ 
\xi\in \fg\longmapsto l_\al(\xi_{\pi(\al)})
\]
with $\xi_{\pi(\al)}$ the tangent vector at $\pi(\al)$ corresponding to the infinitesimal flow in the $\xi$-direction. 

Set $\hat{\fg}^\ast:=\fg^\ast\times_{\ft\ast/W}\ft^\ast$ and consider the fiber product
\[
T^\ast X\times_{\fg^\ast} \hat{\fg}^\ast \cong T^\ast X\times_{\ft^\ast/W}\ft^\ast,
\] where the second fiber product is taken with respect to the composition of the moment map $\mu$ with the Chevalley quotient map 
\[
\chi:\fg^\ast\lra \ft^\ast/W.
\]
In his study of normal $G$-equivariant embeddings $X\subset\overline{X}$, Knop defines a map
\[
\hat{\ka}_X: (X\times \Fl_X)^\circ\times \ft_X^\ast\lra T^\ast X\times_{\ft^\ast/W}\ft^\ast;
\] this is linear in the second factor, so it suffices to define it on the lattice of characters
\[
\Lam_X:= \Hom(T_X,\Gm)\subset \ft_X^\ast.
\]
For $(x,P,\chi)\in (X\times \Fl_X)^\circ\times \Lam_X,$ let $f_\chi$ be a rational $P$-eigenfunction with eigencharacter $\chi,$ which is unique up to scaling. Then Knop's map is defined by
\[
\hat{\ka}_X(x,P,\chi) = \left(\frac{d_xf_\chi}{f_\chi(x)},\chi\right).
\]

While the fiber product is not generally irreducible, the image of the map $\hat{\ka}_X$ singles out a distinguished component 
\[
\widehat{T^\ast X}\subset T^\ast X\times_{\ft^\ast/W}\ft^\ast,
\]
the elements of which are known as \emph{polarized cotangent vectors}. Moreover, if we restrict to the preimage of the regular locus $(\ft_X^\ast)^{reg}$, the resulting open subvariety $\widehat{T^\ast X}^{reg}$ is a Galois cover of an open subset of $T^\ast X$ with Galois group $W_X$ a sub-quotient of $W,$ known as the \emph{little Weyl group} of the variety $X$. Finally, there is a commutative diagram
\begin{equation}\label{eqn: inclusion thing}
\begin{tikzcd}
 T^\ast X\times_{\ft^\ast/W}\ft^\ast\ar[r]&\ft^\ast\ar[r]&\ft^\ast/W\\
 \widehat{T^\ast X}\ar[u]\ar[r]&\ft_X^\ast\ar[u]\ar[r]&\ft_X^\ast/W_X\ar[u],
\end{tikzcd}
\end{equation}
where the vertical arrows are the natural maps.

Now fix a point $x\in X$. There is a natural commutative diagram
\begin{equation}\label{eqn: inclusion thing2}
\begin{tikzcd}
  \widehat{T^\ast_x X}\ar[d]\ar[r]&\ft_X^\ast\ar[d]\\
T^\ast_x X\ar[r]&\ft_X^\ast/W_X.
\end{tikzcd}
\end{equation}
 For any choice $P\in p_1^{-1}(x)$,
\[
\hat{\ka}_X(x,P,-):\ft_X^\ast \lra \widehat{T^\ast_{x}X}
\]gives a section of the top horizontal arrow.

\subsection{The case of a quasi-split symmetric space} Let us now compare the constructions of the previous section to the notions in Section \ref{Section: relative Groth}. To that end, we assume that $X$ is a quasi-split symmetric space associated to $G$. The quasi-split assumption implies that for any Borel subgroup $B\in \Fl_G$, $P(X)= B$ so that $\Fl_X=\Fl_G.$

By the definition of a symmetric space, for any $x\in X$ there is an involution $\theta_x:G\to G$ such that $G^{\theta_x}$ is the stabilizer $G_{x}$, and $X\cong G/G_{x}.$ In this context, $(X\times \Fl_X)^\circ$ is the variety of pairs $(x,B)$ such that $B$ is split with respect to the involution $\theta_x$.

Fix a base point $x_0\in X$ and set $\theta=\theta_{x_0}$ and $G_0=G_{x_0}$. We have the grading
\[
\fg=\fg_0\oplus \fg_1,
\]
where the notation is as in the main body of the paper. Fixing a $\theta$-invariant inner product on $\fg$, which easily exists in characteristic zero (see \cite[Section 3]{Levy} in positive characteristic), we have isomorphisms $\fg\cong \fg^\ast$, and $\fg_0\cong \fg_0^\ast$. Then we see that
\[
T^\ast_{x_0}X= \fg_0^\perp\cong \fg_1.
\]

A key simplifying feature of the symmetric case is the existence of the canonical involution $\theta_{can}: T\to T$ from Section \ref{Section: quasisplit}; see Lemma \ref{Lem: canonical torus match} and Proposition \ref{Prop: canonical involution} for the relevant results. This gives rise to a natural direct sum decomposition of the Lie algebra
\[
\ft\cong \ft_0\oplus \ft_1,
\]
where the $(-1)$-eigenspace $\ft_1\cong \fa:=\ft_X$ is naturally identified with the Lie algebra of the canonical torus $T_X$ of the variety $X$.   Moreover, the maximal $\theta_{can}$-split subtorus $A\subset T$ satisfies 
\[
W_\fa:=W_X\cong W(G,A).
\]
Fixing a $\theta_{can}$-invariant inner product on $\ft$, we have $\ft\cong \ft^\ast$ and the inclusion $\fa\subset \ft$ corresponds to the inclusion $\ft_X^\ast\subset \ft^\ast$ from the previous section.

Considering the fibers over $x_0\in X$, the diagram (\ref{eqn: inclusion thing}) now takes the form
\[
\begin{tikzcd}
 \fg_1\times_{\ft/W}\ft\ar[r]&\ft\ar[r]&\ft/W\\
 \hat{\fg}_1\ar[u]\ar[r]&\fa\ar[u]\ar[r]&\fa/W_\fa\ar[u],
\end{tikzcd}
\]
where $ \hat{\fg}_1=C_1 \subset \fg_1\times_{\ft/W}\ft$ is the connected component of $ \fg_1\times_{\ft/W}\ft$ isolated in Section \ref{subsection: component}. We note that the notation $\hat{\fg}_1$ here is inconsistent with the notation of Section \ref{Section: relative Groth}. We hope this causes no confusion. Also, the diagram (\ref{eqn: inclusion thing2}) corresponds to 
\begin{equation*}
\begin{tikzcd}
 \hat{\fg}_1\ar[d,"\pi"] \arrow[r,"\hat{\chi}_1"] & \fa\ar[d]\\
\fg_1 \arrow[r,"\chi_1"] &\fa/W_\fa,
\end{tikzcd}
\end{equation*}
which extends diagram (\ref{eqn: cartesian1}).

In Section \ref{subsection: component}, we considered sections $\ka_{KW}:\fa\lra \hat{\fg}_1$ of the top horizontal arrow of the form 
\[
\ka_{KW}(a) = (\ka(\overline{a}),a)
\]for a section $\ka:\fa/W_\fa\lra \fg_1$, known as a Kostant-Weierstrass section. The proofs of the various properties of such a section are rather involved; see \cite[Section II.3]{KR71}. Its definition relies on the choice of regular nilpotent element $e\in \fg_1$ and $r_1=\dim(\fa)$-dimensional affine subspace $e+\mathfrak{v}\subset \fg_1$. For any such choices, the section $\ka$ has the properties that
\begin{enumerate}
\item the image of $\ka$ is contained in $\fg_1^{reg}$
\item $\ka(\fa^{reg}/W_\fa)$ meets each regular semi-simple $G_0$-orbit in exactly one point.
\end{enumerate}

Comparing with the sections introduced in the preceding section, we claim that $\ka_{KW}$ is not given by $\hat{\ka}_X(x_0,P,-)$ for any $\theta$-split Borel subgroup $P$. To see this, just note that for any such $P$,
\[
\hat{\ka}_X(x_0,P,0)=(0,0)\in \widehat{T^\ast_{x_0}X} \cong \hat{\fg}_1,
\]
since in this case $f_\chi=f_1$ is a $P$-fixed function and the sphericity of $X$ implies that any $P$-fixed function is constant. On the other hand, 
\[
\ka_{KW}(0) = (e,0)
\]
where $e$ is our chosen regular nilpotent element. More generally, if we let $\mathring{\mathcal{F}}l_X\subset \Fl_X$ denote the open subvariety of $\theta$-split Borel subgroups, then  $p^{-1}_1(x_0)\cong \mathring{\mathcal{F}}l_X$ and the Kostant-Weierstrass section does not factor through 
\[
\hat{\ka}_{x_0}:p_1^{-1}(x_0)\times \fa\lra \widehat{T^\ast_{x_0}X}
\]where $\hat{\ka}_{x_0}=\hat{\ka}_X({x_0},-,-)$, in the sense that there is no morphism $\iota:\fa\to  \mathring{\mathcal{F}}l_X\cong p^{-1}_1(x_0)$ such that
\[
\ka_{KW}=\hat{\ka}_{x_0}\circ \iota\boxtimes I_\fa.
\]

Our analysis in Section \ref{Section: relative Groth} implies that such a map exists over the regular semi-simple locus.
\begin{Lem}
Fix a Kostant-Weierstrass section $\ka_{KW}:\fa\lra \hat{g}_1$. There exists a unique morphism
\[
\iota_\ka:\fa^{reg}\lra \mathring{\mathcal{F}}l_X
\]
such that the diagram
\[
\begin{tikzcd}
 \fa^{reg}\ar[rd,"\ka_{KW}",swap]\ar[r,"\iota_\ka\boxtimes I_\fa"]&\mathring{\mathcal{F}}l_X\times \fa^{reg}\ar[d,"\hat{\ka}_{x_0}"]&\\
 &\widehat{T^\ast_{x_0}X},
\end{tikzcd}
\]
where $\hat{\ka}_{x_0}=\hat{\ka}_X({x_0},-,-)$ is the restriction of Knop's map to the fiber $\mathring{\mathcal{F}}l_X\times \fa$ over $x_0.$
\end{Lem}
\begin{proof}
This follows from Proposition \ref{Prop: regular characterization}, which characterizes $\hat{\fg}_1^{reg}\cong \fgreg\subset \fg_1\times\Fl_X$. Indeed the map $\iota_\ka$ is given by composing $\ka_{KW}$ with the projection $\fgreg\lra \Fl_X$.
\end{proof}

\end{appendices}



\bibliographystyle{alpha}

\bibliography{bibs}
\end{document}